\documentclass[12pt,reqno]{amsart}

\usepackage{amstext, amssymb, amsthm, amsfonts,bbm,tikz,amsmath}
\usepackage{mathtools}

\usepackage{enumerate}
\usepackage[T1]{fontenc}
\usepackage[version=4]{mhchem}
\usepackage[normalem]{ulem}
\usepackage{chemfig}
 
\usepackage{comment}
\usepackage{hyperref}
\usepackage{geometry}

\let\emptyset\varnothing

\newtheorem{theorem}{Theorem}[section]
\newtheorem{lemma}[theorem]{Lemma}
\newtheorem{corollary}[theorem]{Corollary}
\newtheorem{proposition}[theorem]{Proposition}

\theoremstyle{definition}
\newtheorem{definition}[theorem]{Definition}
\newtheorem{example}[theorem]{Example}
\newtheorem{condition}{Condition}

\newcommand{\cI}{\mathcal{I}}

\newcommand{\mR}{\mathcal{R}}

\newcommand{\cV}{\mathcal{V}}
\newcommand{\kR}{\mathfrak{R}}
\newcommand{\inc}{\mathrm{inc}}
\newcommand{\out}{\mathrm{out}}
\newcommand{\ter}{\mathrm{ter}}

\newcommand{\1}{\mathbf{1}}
\newcommand{\R}{\mathbb{R}}
\newcommand{\N}{\mathbb{N}}
\newcommand{\Z}{\mathbb{Z}}
\newcommand{\cC}{\mathcal{C}}
\newcommand{\cR}{\mathcal{R}}

\newcommand{\cS}{\mathcal{S}}
\newcommand{\cH}{\mathcal{H}}
\newcommand{\cX}{\mathcal{X}}
\newcommand{\cY}{\mathcal{Y}}
\newcommand{\cZ}{\mathcal{Z}}
\newcommand{\cU}{\mathcal{U}}
\newcommand{\cE}{\mathcal{E}}
\newcommand{\cN}{\mathcal{N}}

\newcommand{\cG}{\mathcal{G}}

\newcommand{\cD}{\mathcal{D}}
\newcommand{\reac}{\mathrm{reac}}
\newcommand{\prdt}{\mathrm{prod}}
\newcommand{\init}{\mathrm{init}}
\newcommand{\supp}{\mathrm{supp}}

\newcommand{\pro}{\mathrm{pro}}

\allowdisplaybreaks
\numberwithin{equation}{section}

\begin{document}

\title[Asymptotic of stationary distributions]{Asymptotic analysis for stationary distributions of scaled reaction networks}
\hfill\break

\author{
Linard Hoessly, Carsten Wiuf and Panqiu Xia}
\address{Data Center of Swiss Transplant Cohort Study, University hospital Basel, Switzerland}
\email{\href{mailto:linard.hoessly@hotmail.com}{linard.hoessly@hotmail.com},}

\address{Department of Mathematical Sciences, University of Copenhagen, Denmark}
\email{\href{mailto:wiuf@math.ku.dk}{wiuf@math.ku.dk}}
\address{Department of Mathematics and Statistics, Auburn University, USA}
\email{\href{mailto:pqxia@auburn.edu}{pqxia@auburn.edu}}
\date{}
\subjclass[2010]{}

\begin{abstract}
We study stationary distributions in the context of stochastic reaction networks. In particular, we are interested in complex balanced reaction networks and reduction of such networks by assuming a set of species (called non-interacting species) are degraded fast (and therefore essentially absent in the network), implying some reaction rates are large compared to others. Technically, we assume these reaction rates are scaled by a common parameter $N$ and let $N\to\infty$. The limiting stationary distribution as $N\to\infty$ is compared to the stationary distribution of the reduced reaction network obtained by algebraic elimination of the non-interacting species.
In general, the limiting stationary distribution might differ from the stationary distribution of the reduced reaction network. We identify various sufficient conditions for when these two distributions are the same, including when the reaction network is detailed balanced  and  when the set of non-interacting species consists of  intermediate species. In the latter case, the limiting stationary distribution essentially retains the form of the complex balanced distribution.   This finding is particularly surprising given that the reduced reaction network  might be non-weakly reversible and exhibit unconventional kinetics.
\end{abstract}


\maketitle

\section{Introduction}

The theory of reaction networks can be used to model the complex behaviour of  chemical systems \cite{anderson2015stochastic,Feinberg}. Ordinary differential equations might be utilized to describe the   evolution of species concentrations when the  molecule counts in the system are relatively large. However, when the molecule counts are low, random fluctuations become significant, and probabilistic techniques might be  employed \cite{mcquarrie}.  Here, continuous-time Markov chains (CTMCs) are commonly applied to model the stochastic  dynamics of the molecule counts.

 A first objective in the study of stochastic    reaction networks is to investigate the stationary distributions of the CTMC, which describe the long-term behaviour of the system  \cite[etc]{non-stand_1,anderson2,Cappelletti,HWX21+,hong2021derivation}. In the present paper, we explore conditional and  asymptotic properties of stationary distributions of multiscaled stochastic reaction networks.  Multiscaled stochastic reaction networks are of relatively recent interest \cite[etc]{CW16,HW2021,KK13,PP15}, and build further on the seminal work by Kurtz \cite{kurtz2}, which revealed a   correspondence between a sequence of scaled stochastic reaction networks and a (limiting) deterministic reaction network.    In particular, we focus on complex balanced stochastic reaction networks with non-interacting species (to be defined later) assuming a sequence of one-parameter scaled stochastic reaction networks. We consider the limiting stationary distribution of the one-parameter scaled stochastic reaction networks and compare it with the stationary distribution of the reaction network obtained by elimination of non-interacting species (to be defined later).
 It is well-known that the limiting stationary distribution and the stationary distribution obtained by elimination of non-interacting species might not be the same  \cite[Example 5.4]{CW16}, \cite[Example 5.8 \& Example 5.8]{HW2021}.

To motivate our work, we consider a reaction network with stochastic mass-action kinetics, 
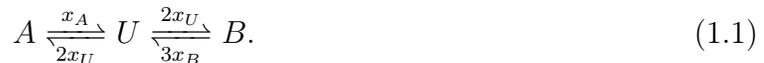
\begin{align}\label{exp_1}
 A\ce{<=>[$x_A$][$2 x_U$]} U \ce{<=>[$2 x_U$][$3 x_B$]} B.
\end{align}
Here, $A$, $B$, and $U$ represent distinct species. $x_A$, $x_B$ and $x_U$ correspond to the molecules counts in the system of species $A$, $B$, and $U$. Additionally, the expressions $x_A$, $2x_U$, and so forth, are the intensity functions for the reactions $A \ce{->} U$, $U \ce{->} A$, etc. These expressions serve to quantify the potential occurrence or propensity of the reactions within the system.  
 The set  $\Gamma = \{x_A + x_B + x_U = T\}$, $T\in\N_0$, is an irreducible component of the CTMC. Due to a well-known result   {\cite[Theorem 4.1]{anderson2}}, this reaction network has a unique Poisson-product form stationary distribution on $\Gamma$ given by
\begin{align*}
\pi(x_A,x_B,x_U) \coloneqq M \frac{6^{x_A}2^{x_B}3^{x_U}}{x_A ! x_B ! x_U!},\qquad (x_A, x_B, x_U)\in\Gamma,
\end{align*}
where $M$ is a positive constant such that $\pi$ is a probability distribution. By  applying a scaling factor $N$ to the reaction rates of $U\ce{->}A$ and $U\ce{->}B$, respectively, we obtain a one-parameter sequence of stochastic reaction networks,
\begin{align*}
\pi_N (x_A,x_B,x_U) \coloneqq  M_N \frac{6^{x_A} 2^{x_B} (3/N)^{x_U}}{x_A! x_B! x_U!} = M_N N^{- x_U} \frac{6^{x_A}2^{x_B}3^{X_U}}{x_A! x_B! x_U!},\qquad (x_A, x_B, x_U)\in\Gamma,\quad N\in\N,
\end{align*}
where $M_N>0$ is a constant depending on $N$. Letting $N\to \infty$, one finds that $\pi_N$ converges pointwise to a probability distribution $\pi_0$ on $\Gamma$ given by
\begin{align*}
\pi_0 (x_A, x_B) \coloneqq \begin{dcases}
\frac{1}{\pi (\Gamma_0)} \pi (x_A,x_B), & (x_A, x_B)\in \Gamma_0 \coloneqq \{x_A + x_B = T\},\\
0, & \mathrm{otherwise}.
\end{dcases}
\end{align*}
This distribution $\pi_0$ is indeed the unique stationary distribution of the  reaction network on $\Gamma_0$ with stochastic mass-action kinetics:
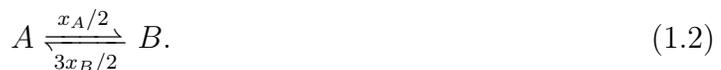
\begin{align}\label{exp_toyred}
  A\ce{<=>[$x_A/2$][$3x_B/2$]} B.
\end{align}
Moreover, the reaction network \eqref{exp_toyred} might be seen as a  reduced stochastic reaction network   by elimination of the species $U$, in the  sense of  \cite{HW2021,hoessly2021sum}.

This motivating example   raises the natural question of how general these observations are. We will address this in the context of weakly reversible reaction networks with a complex balanced distribution and a set of non-interacting species (to be defined within the paper).  The species $U$ in the example is a non-interacting species.

In this context, we will show that the conditional distribution is a complex balanced distribution on a reduced stochastic reaction network obtained by eliminating the non-interacting species \cite{HW2021,hoessly2021sum}. Furthermore, this reduced reaction network might be considered the natural limit of a one-parameter sequence of scaled reaction networks. Most proofs are given in a separate section.

 We provide a brief overview of key definitions pertaining to non-interacting species and   reduced reaction networks. The concept of non-interacting species  was initially introduced in \cite{Variable_el} in the context of variable elimination in deterministic reaction networks, extending the notion of intermediate species  \cite{CW16,intermediates}. Non-interacting species are found in  most realistic  models of   chemical reaction networks and are generally conceived of as physically short-lived molecules.
In the stochastic setting, the connection between   reduced reaction networks (to be made precise below) and stochastic reaction networks with linear scaled kinetics was investigated in \cite{HW2021}, and it was established   that under appropriate assumptions, the dynamics     converges (in some sense) to that of the  reduced reaction network (see also \cite{hoessly2021sum}).

\subsection*{Acknowledgements}
LH was partially supported by the Swiss National Science Foundations Early Postdoc.Mobility grant (P2FRP2 188023). PX was partially supported by U.S. National Science Foundation grant DMS-2246850. The work presented in this article was supported by Novo Nordisk Foundation (Denmark), 
grant NNF19OC0058354.

\section{Preliminaries}

\subsection{Notation}

Let $\R$, $\R_{\geq 0}$ and $\R_{>0}$ be the set of real, non-negative and positive numbers, receptively.  Let $\N$ and $\N_0$ the set of positive and non-negative integers, receptively, and $\Z$ the set of integers.   
For $x=(x_1,\dots, x_n), y=(y_1,\dots, y_n)\in \R^n$, we say $x\geq y$, if $x_i\geq y_i$ for all $i=1,\dots, n$; and $x>y$ if $x\geq y$ and $x\neq y$.  Additionally, we define $x\vee y= (\max\{x_1,y_1\},\dots, \max\{x_n, y_n\})$, and $\langle x,y\rangle=\sum_{i=1}^n x_iy_i$ for the inner product on $\R^n$. For a finite set $\cI$, $|\cI|$ is the cardinality of $\cI$.

\subsection{Reaction networks}

A \emph{reaction network}  is a triple $\cN=(\cS,\cC,\cR)$  of three finite sets: a set $\cS=\{S_1,\ldots,S_n\}$ of \emph{species}, a set $\cC$ of \emph{complexes}, which  are non-negative linear combination of the species (complexes are identified as elements of $\N_0^n$), and a set  $\cR$ of \emph{reactions}, which are elements of $\cC\times\cC$. One might consider $(\cC,\cR)$ as a digraph. We assume all species have a positive coefficient in some complex and all complexes are the source or target of some reaction. If this case,  $\cC$ and $\cS$ can be determined from $\cR$. In examples, we simply draw the reactions.

For simplicity, for a reaction $r\in\cR$, we write  $r=(y,y')$ or $r=y\ce{->}y'$, where $y$ is the source   of $r$ and $y'$ the target. In the chemical literature, these are called the \emph{reactant} $\reac (r)=y$ and the \emph{product} $\prdt (r)=y'$ of $r$, respectively.  The \emph{reaction vector} of $r$ is $\zeta_r= y'-y\in\Z^n$.

We introduce an associative binary operation \cite{hoessly2021sum}:
\begin{align}\label{def_sum}
\oplus & \colon \ \big(\N_0^n \times \N_0^n\big) \times \big(\N_0^n \times \N_0^n\big) \to \N_0^n \times \N_0^n, \\
(y_1, y_1')\oplus (y_2, y_2') & \coloneqq  \big(y_1+0\vee (y_2-y_1'), y_2'+0\vee (y_1'-y_2)\big)   \nonumber
\end{align}
for $(y_1, y_1'), (y_2, y_2') \in \N_0^n \times \N_0^n$.
It might be interpreted as the sum of two reactions, for example, 
\[\big( (1,1,0), (0,0,1) \big) \oplus \big( (0,1,1), (0,1,0) \big) = \big( (1,2,0), (0,1,0) \big),\]
 corresponds to the sum (cf. \cite[Section 3.1.2]{syst_ing}),
\[S_1 + S_2 \ce{->} S_3 \quad \oplus \quad S_2 + S_3 \ce{->} S_2 \quad = \quad S_1 + 2S_2 \ce{->} S_2,\]
The left hand side of the sum consists of the molecules required for the two   reactions to proceed in succession, and the right hand side consists of the molecules that are produced and not consumed in the two reactions.

In the following, we assume a reaction network $\cN=(\cS,\cC,\cR)$ is given.

\begin{definition}
A state   $x_1\in \N_0^n$ \emph{leads to} $x_2\in \N_0^n$ via   $\cN$, denoted by $x_1 \to_{\cN} x_2$, if there exists reactions $r_1,\dots, r_k \in \cR$, such that  $r_1\oplus \cdots \oplus r_k=(y,y')$ with $x_1\geq y$ and $x_2-x_1=y'-y$.
\end{definition}

\begin{definition}[{\cite[Definition 10]{Cappelletti}}]\label{def_irdcmp}
A set $\Gamma\subseteq \N_0^n$ is   an irreducible component of $\cN$, if the for all $x\in \Gamma$ and all $z\in \N_0^n$, $x\to_{\cN} z$ if and only if $z\in \Gamma$.
\end{definition}

\begin{definition}
A reaction network is 
 \begin{enumerate}
\item[i)] {\it reversible}, if  $y\ce{->} y'\in \cR$ implies  $y'\ce{->}y\in \cR$.

\item[ii)] {\it weakly reversible}, if for any  $y\ce{->}y'\in \cR$, there exists a sequence   $r_1,\dots, r_m\in \cR$, satisfying  $\reac (r_1)=y'$, $\prdt(r_i)=\reac(r_{i+1})$ for all $i=1,\dots, m-1$ and $\prdt(r_m)=y$.
\end{enumerate}
\end{definition}

A {\it stochastic reaction network} (SRN) models the stochastic dynamics of a reaction network. Specifically,  the evolution of the molecule counts, $(X(t)\colon t\in \R_{\ge0})$, is modelled as an $\N_0^n$-valued CTMC, satisfying the following stochastic equation:
\begin{align}\label{eq:repr}
X(t)=X(0)+\sum_{r\in \cR}Y_r\Big(\int_0^t\lambda_{r}(X(s))ds\Big)\zeta_r,
\end{align}  
where $(Y_{r}\colon r\in \cR)$ is a collection of i.i.d.  unit rate Poisson processes   and $\lambda=(\lambda_{r}\colon r \in \cR)$ is a collection of non-negative intensity  functions on $\N_0^n$, known as the {\it  kinetics}.   {\it Stochastic mass-action kinetics} is commonly assumed in the literature; however, it is not a prerequisite for the majority of this paper's content. For any reaction $r = y_r \to y_r'$, the intensity function under mass-action kinetics can be written as the following form
\[
\lambda_r (x) = \kappa_r \frac{x!}{(x - y_r)!} = \kappa_r \prod_{i=1}^n \frac{x_i !}{(x_i - y_{r,i})!},
\]
with some positive constant $\kappa_r$.

An SRN is denoted by $(\cN,\lambda)$. Throughout, we assume a  compatibility condition is fulfilled:

\begin{condition}\label{con_cmp}
For any $r=y\ce{->}y'\in \cR$, $\lambda_{r}(x)>0$ if any only if $x\geq y$.
\end{condition}

If Condition \ref{con_cmp} is fulfilled, then $x\to_{\cN} z$ implies that  there are positive probability to go from $x$ to $z$, and vice versa.

A probability measure $\pi$ on an irreducible component $\Gamma\subseteq \N_0^n$ of $\cN$ is   \cite[Definition 4.1]{JoshiCap}:
\begin{enumerate}
\item[i)] a {\it stationary distribution}, if for all $x\in \Gamma$,
\begin{align}\label{def_msteq}
\pi(x)\sum_{r\in \cR}\lambda_{r}(x)=\sum_{r\in \cR}\pi(x-\zeta_r)\lambda_{r}(x-\zeta_r);
\end{align}

\item[ii)] a {\it complex balanced distribution}, if for all complexes $\eta\in \cC$, and all  $x\in \Gamma$,
\begin{align}\label{def_cbd}
\pi(x)\sum_{y'\colon \eta\to y'\in \cR}\lambda_{\eta\to y'}(x)=\sum_{y\colon y\to \eta\in \cR}\pi(x+y-\eta)\lambda_{y\to \eta}(x+y-\eta);
\end{align}

\item[iii)] a {\it detailed balanced distribution}, if $\cN$ is reversible and for all $y\ce{->}y'\in \cR$,
\begin{align*}
\pi(x)\lambda_{y\to y'}(x)=\pi(x+y'-y)\lambda_{y'\to y}(x+y'-y).
\end{align*}
\end{enumerate}

Note that we reserve `stationary distribution' to a distribution on a single irreducible component. Hence,  using Definition \ref{def_irdcmp}, a stationary distribution $\pi$  on an irreducible component $\Gamma$ of $(\cN,\lambda)$ fulfils  $\pi(x)>0$ for all  $x\in \Gamma$.
Equation \eqref{def_msteq} is known as the \emph{master equation} in the chemical literature (also known as the \emph{Kolmogorov forward equation}) for the SRN $(\cN,\lambda)$ \cite{anderson2015stochastic}. A detailed balanced distribution is   a complex balanced distribution, and a complex balanced distribution is  a stationary distribution.

\subsection{Non-interacting species and  reduction}

Let $\cN=(\cC,\cR,\cS)$  be a reaction network, and $\cU = \{U_1,\dots, U_m\}$ be a proper subset of $\cS$. Order the species such that   $ \cS \setminus \cU = \{S_1, \dots, S_{n-m}\}$ and $\cS=\{S_1, \dots, S_{n-m}, U_1,\dots, U_m\}$. Let $x=(z,u)\in\N_0^{n-m}\times \N_0^m=\N_0^n$ with $z=(z_1, \ldots, z_{n-m})$ and $u=(u_1,\ldots,u_m)$. We introduce the projection $\rho\colon\N_0^n \to \N_0^m$ onto the space of the species in $\cU$  by
\begin{align}\label{def_rho}
\rho (x) = \sum_{i=1}^m u_i U_i \in \N_0^m,\quad \text{for } x =(z,u) \in \N_0^n.
\end{align}

\begin{definition}
The set $\cU$ is a set of {\it non-interacting species}, if for any $y\in \cC$, either $\rho  (y) = 0$ or $\rho(y) = U_i$ for some $U_i \in \cU$.  The complement $\cS\setminus\cU$ is the set of \emph{core species}.
\end{definition}

For a set  $\cU=\{U_1,U_2,\dots, U_m\}$ of non-interacting species, let  $\{\cC_0, \cC_1, \dots, \cC_m\}$ be the partition of $\cC$  by 
\begin{align*}
 \cC_i \coloneqq \begin{cases} 
\displaystyle \big\{y\in \cC\colon \rho (y) = 0 \big\}, & i = 0,\\
\displaystyle \big\{y \in \cC\colon \rho (y) = U_i \big\}, &1\leq i\leq m,
\end{cases}
\end{align*}
and the partition  $\{\cR_{i,j}\colon i,j=0,\dots, m\}$  of $\cR$, where 
\[\cR_{i,j} \coloneqq \big\{y\ce{->}y'\colon y\in \cC_i, y'\in \cC_j \big\},\quad i,j=0,\dots, m.\] 
 We introduce a \emph{multi-digraph} $(\cV,\cE)$  \cite{saez2}:
\[
\cV=\{U_0\}\cup \cU\quad \mathrm{and} \quad \cE=\big\{U_i\ce{->[$r$]}U_j\colon r\in \cR_{i,j}, i,j=0,\dots, m\big\},
\]
 and note that there is a one-to-one correspondence between the edges $U_i\ce{->[$r$]}U_j$ and the reactions $r \in \cR_{i,j}$. 

   A $q$-step \emph{walk} $\theta$ in $(\cV,\cE)$ is a sequence of $q\in\N$ edges 
 \begin{align*}
\theta = U_{i_0}\ce{->[$r_1$]} U_{i_1} \ce{->[$r_2$]} \cdots \ce{->[$r_q$]} U_{i_q},
\end{align*}
where $r_k\in \cR_{i_{k-1},i_k}$. A $q$-step walk is \emph{closed}  if $U_{i_0} = U_{i_q}$. If $U_{i_0}=U_0$, then $U_{i_q}$ is said to be \emph{produced}, and if $U_{i_q}=U_0$, then $U_{i_0}$ is said to be \emph{degraded}. Let $\cU_{\pro}$ and $ \cU_{\deg}$ be the sets of produced and degraded non-interacting species, respectively.

For  $q\in \N_0$ and $i\in\{0,\dots, m\}$, we let $\Xi_i(q)$ be the collection of all $(q + 1)$-step  closed walks, that starts and ends at $U_i$ without passing though $U_0$. Then, every element in $\Xi_i(q)$ is of the form
\begin{align}\label{def_gm}
\gamma = U_i\ce{->[$r_0$]}U_{i_1}\ce{->[$r_1$]}\cdots \ce{->[$r_{q-1}$]}U_{i_q}\ce{->[$r_q$]}U_i,
\end{align}
where $i_1,\dots, i_q \in \{1,\dots, m\}$, $r_j \in \cR_{i_j,i_{j+1}}$, $0\leq j\leq q$, and by definition $i_0 = i_{q+1}= i$. In particular, if $\gamma\in\Xi_i(0)$ then $\gamma$ consists of one edge $U_i \ce{->[$r$]} U_i$ only, and if  $i=0$, then $\gamma$ is a reaction in $\cR$ between core species only.   Furthermore, define
\[\Xi_i \coloneqq \cup_{q=0}^{\infty}\Xi_i(q),\quad i=0,\ldots,m. \]

\begin{example}
Consider the reaction network \cite[Section 3.1.2]{syst_ing},
\begin{align}\label{ex_eab}
  E + A \ce{<=>} EA, \quad EA + B \ce{<=>} EAB, \quad EAB \ce{->} E + P + Q.
\end{align}
It models the conversion of two substrates $A,B$ into two other substrates $P,Q$ by means of an enzyme $E$, and short-lived intermediate molecules.
Let $\cU = \{U_1, U_2\} = \{EA, EAB\}$ be a set of non-interacting species. Then,   
\begin{gather*}
\cC_0 = \{E+A, E + P + Q\},\quad \cC_1  = \{EA, EA + B\},\quad \cC_2  = \{EAB\},\\
\cR_{0,0} = \cR_{0,2} = \cR_{1,1} = \cR_{2,2} = \emptyset, \quad \cR_{0,1} = \{E + A \ce{->} EA\},\\
 \cR_{1,0} = \{EA \ce{->} E + A\}, \quad \cR_{1,2} = \{EA + B \ce{->} EAB \},\\
\cR_{2,0} = \{EAB \ce{->} E + P + Q\}, \quad    \cR_{2,1} = \{EAB \ce{->} EA + B\}.
\end{gather*}
The associated multi-digraph is:
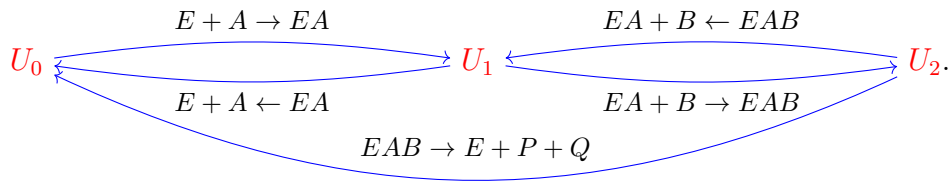
\begin{center}
\begin{tikzpicture}[node distance=6cm]
  \node (U0) {\color{red}$U_0$};
  \node (U1) [right of=U0] {\color{red}$U_1$};
  \node (U2) [right of=U1] {{\color{red}$U_2$}.};
  
  \path (U0) edge [->,bend left=8, blue] node [above,black] {\footnotesize{$E + A\to EA$}} (U1);
  \path (U1) edge [->,bend left=8, blue] node [below,black] {\footnotesize{$E + A \gets EA$}} (U0);
  \path (U2) edge [->, bend right=8,blue] node [above,black] {\footnotesize{$EA + B \gets EAB$}} (U1);
  \path (U1) edge [->, bend right=8,blue] node [below, black] {\footnotesize{$EA + B \to EAB$}} (U2);  
  \path (U2) edge [->, bend left=25, blue] node [above=3, black] {\footnotesize{$EAB \to E + P + Q$}} (U0);
\end{tikzpicture}
\end{center}
We find 
\begin{gather*}
\Xi_1(0) = \emptyset,\quad \Xi_1(1) = \big\{U_1\ce{->[$EA + B \to EAB$]} U_2 \ce{->[$EAB \to EA + B$]} U_1\big\},\quad \Xi_1(2) = \emptyset,\\
\Xi_1(3) = \big\{U_1 \ce{->[$EA + B \to EAB$]} U_2 \ce{->[$EAB \to EA + B$]} U_1 \ce{->[$EA + B \to EAB$]} U_2 \ce{->[$EAB \to EA + B$]} U_1\big\}.
\end{gather*}
\end{example}

\begin{definition}\label{def_elmnb}
Let   $\cU$ be a set of non-interacting species. If $\cU_{\pro} \subseteq \cU_{\deg}$, then $\cU$ is said to be \textit{eliminable}.
\end{definition}

 In  \eqref{ex_eab}, we have $\cU_{\pro} = \cU_{\deg} = \{U_1, U_2\}$, and thus  $\cU$ is eliminable.
Definition \ref{def_elmnb}   differs   from \cite[Definition 4.1]{hoessly2021sum}  with the latter encompassing more general situations. However, these two definitions align in the context of non-interacting species.

\begin{lemma}\label{lmm_i=j}
Let $\cN$ be  weakly reversible  and $\cU$ a set of non-interacting species. Then, $\cU_{\pro} = \cU_{\deg}$ and  $\cU$ is eliminable. 
\end{lemma}

\begin{proof}
 To prove $\cU_{\pro} = \cU_{\deg}$, it suffices to prove the one-directional inclusion  $\cU_{\pro} \subseteq \cU_{\deg}$. The reverse direction follows by the same reasoning. Let $U \in \cU_{\pro}$. Then, there exists $q \in \N_{\geq 0}$ and a walk   in $(\cV, \cE)$
\[
  U_0 \ce{->} U_{i_1} \ce{->} \cdots \ce{->} U_{i_q} \ce{->} U_i,
\]
where the edge labels  are skipped for simplification. If $q = 0$, then $U_0 \ce{->} U_i \in \cE$, and thus $U_i \in \cU_{\deg}$ by  weak reversibility.

For  $q \geq 1$, we   prove the lemma by induction. First, $U_{i_q} \ce{->} U_i$ implies the existence of a reaction $y \ce{->} y' \in \cR_{i_q, i}$. Then, due to  weak reversibility, there exists a sequence of reactions in $\cR$, such that
\[
   y' \ce{->} y_1\ce{->}\cdots \ce{->} y_{q'}\ce{->} y.
\]
If there exists $j\in \{1,\dots, q'\}$ such that $y_j \in \cC_0$, then we are done. Otherwise, by the induction hypothesis, we have $U_{i_q} \in \cU_{\deg}$. This implies that there exists a walk of the form
\[
  U_0 \ce{->} \cdots \ce{->} U_{i_q} \ce{->} U_i \ce{->} \cdots \ce{->} U_{i_q} \ce{->} \cdots \ce{->} U_0.
\]
Therefore, $U_i \in \cU_{\deg}$. The proof is complete.
\end{proof}
 
We introduce a map $\kR$ defined on the collection of all walks in $(\cV,\cE)$:
 \begin{align*}
\kR(\theta) &  \coloneqq r_1\oplus r_2 \oplus \cdots \oplus r_q \in \N_0^n \times \N_0^n,\\
\theta &= U_{i_1}\ce{->[$r_1$]}U_{i_2}\ce{->[$r_2$]}\cdots\ce{->[$r_q$]}U_{i_{q+1}}.
\end{align*}
Then, we have  
\[
\cR_{i,i} = \kR(\Xi_i(0)) \subseteq \kR (\Xi_i), \quad i=0,\dots, m.\]
In particular, $\rho (y) = \rho  (y') = 0$ for  $\gamma \in \Xi_0$ with $ \kR(\gamma)=(y,y')$.  

\begin{definition}\label{def_grrn}
Let $\cU$ be an eliminable set of species.
The \emph{reduced reaction network} $\cN_{\cU}$, obtained by elimination of $\cU$ from  $\cN$,  is the triple $(\cS_{\cU},\cC_{\cU},\cR_{\cU})$ given by
\[\cS_{\cU} \coloneqq \mathop\cup_{y\in \cC_{\cU}} \supp(y)\subseteq \cS_0,\quad  \cC_{\cU} \coloneqq \{y,y'\colon y\to y'\in\cR_{\cU}\}, \quad   \cR_{\cU} = \kR(\Xi_0) \setminus \{(y,y)\colon y\in \N_0^{n_{\cU}}\},\]
where  $\supp (y) = \{S_i\colon y_i>0\}$ is the \emph{support} of a complex $y = \sum_{j=1}^n y_j S_j$, and
$n_{\cU} = |\cS_{\cU}|$.
\end{definition}

By Definition \ref{def_grrn},  we find for   \eqref{ex_eab},
\[
\kR(\Xi_0) = \bigg\{\begin{array}{c}
    E + A \ce{->} E + A, \quad E + A + B \ce{->} E + A + B, \\
    E + A + B \ce{->} E + P + Q
  \end{array} \bigg\},
\]
and thus the reduced reaction network is the triple $(\cS_{\cU},\cC_{\cU}, \cR_{\cU})$ with $\cS_{\cU} =\{E, A, B, P, Q\}$, $\cC_{\cU} = \{E + A + B, E + P + Q\}$, and $\cR_{\cU}  = \{E + A + B \ce{->} E + P + Q\}$.

The set $\cR_{\cU}$  may contain infinitely many reactions \cite[Example 2]{hoessly2021sum}, and thus the reduced reaction network might not be a reaction network.   Condition \ref{con_1} below is equivalent to $\cR_{\cU}$ being finite \cite{hoessly2021sum}, and $\cN_\cU$ is thus a reaction network if  Condition \ref{con_1} is fulfilled. The proof is in Section \ref{app_fgrn}.

\begin{condition}\label{con_1} 
If $\gamma\in \Xi_i$, for some $i=1,\ldots,m$, then $\kR(\gamma) = (y,y)$ for some $y\in \N_0^n$.
\end{condition}

\begin{lemma}\label{lmm_fgrn}
The set  $\cR_{\cU}$ is finite if and only if Condition \ref{con_1} holds.
\end{lemma}

Let $(\cN,\lambda)$ be an SRN. We introduce a kinetics $\lambda_{\cU}$ on the reduced reaction network $\cN_{\cU}$ \cite{HW2021}. For  $r \in \cR_{i,j}\subseteq \cR$, $i,j\in \{0,\dots, m\}$,   define the function $\beta_r \colon \N_0^{n}\to [0,1]$  by
\begin{align}\label{def_btijk}
\beta_r (x) \coloneqq \lambda_r(x) \bigg/ \sum_{k = 0}^m\sum_{r'\in  \cR_{i,k}}\lambda_{r'}(x),
\end{align}
with the convention $0/0 \coloneqq 0$. Then,  
\begin{align}\label{equ_sbt=1}
\sum_{k=0}^m \sum_{r\in \cR_{i,k}} \beta_r(x) = \1_{\cD_i}(x),\quad i=0,\ldots,m,
\end{align}
where
\[
\cD_i \coloneqq \bigcup_{k=0}^m \bigcup_{r\in \cR_{i,k}} \big\{x\in \N_0^n\colon x\geq \reac (r)\big\}.
\]
For  $\gamma\in \Xi_0$,   define $\lambda^*_{\cU,\gamma}\colon\N_0^{n}\to \R_{\geq 0}$   by  
\begin{align}\label{def_lmd0}
\lambda^*_{\cU,\gamma}(x) \coloneqq \lambda_{r_0}(x) \prod_{j=1}^{q}\beta_{r_j} \bigg(x+\sum_{i=1}^{j}\zeta_{i - 1}\bigg), 
\end{align}
where $\zeta_{i- 1} \coloneqq \zeta_{r_{i- 1}}$ is the reaction vector of $r_{i - 1}$, and $i_0 = i_{q+1} \coloneqq 0$.\
Define the kinetics $\lambda_{\cU} = (\lambda_{\cU,r} \colon r\in \cR_{\cU})$ by
\begin{align}\label{def_lmd'}
\lambda_{\cU,r}(x)= \sum_{\gamma\in \Xi_0 \colon \kR(\gamma) = r}\lambda^*_{\cU,\gamma}(x),\quad x\in \N_0^{n_{\cU}},\quad r\in\cR_\cU.
\end{align}
It holds that $\lambda_{\cU,r}(x)> 0$ if and only if $x\geq \reac(r)$ \cite[Corollary 3.11]{HW2021}.
As a consequence,  $(\cN_{\cU},\lambda_{\cU})$ satisfies Condition \ref{con_cmp}.  Furthermore, using Definition \ref{def_grrn}, we have
$$ \sum_{r\in\cR_\cU} \lambda_{\cU,r}(x)\le \sum_{\gamma\in\Xi_0} \lambda^*_{\cU,\gamma}(x)=\sum_{r\in\cR} \lambda_r(x)  <\infty$$
\cite[Lemma 3.10]{HW2021}. This implies that the evolution of the molecule counts of the core species, $(Z(t)\colon t\in \R_{\ge0})$,  satisfies the following stochastic equation, similar to \eqref{eq:repr}:
\begin{align*}
Z(t)=Z(0)+\sum_{r\in \cR_\cU}Y_r\Big(\int_0^t\lambda_{r}(Z(s))ds\Big)\zeta_r,
\end{align*}  
where $(Y_r$, $r\in\cR_\cU)$, is a collection (possibly infinite) of unit-rate i.i.d. Poisson processes.

\begin{definition}
The pair $(\cN_{\cU},\lambda_{\cU})$, defined in Definition \ref{def_grrn} and \eqref{def_lmd'}, is the \emph{reduced SRN} of $(\cN,\lambda)$ (obtained by elimination the non-interacting species in $\cU$). 
\end{definition}

If $\cR_\cU$ is finite then $(\cN_{\cU},\lambda_{\cU})$ is an  SRN. The kinetics for the reduced reaction network in Example  \eqref{ex_eab}  is 
\begin{align*}
\lambda_{\cU, E + A + B \to E + P + Q} (x) = & \lambda_1(x) \lambda_2(x) \lambda_3(x) \lambda_4(x) 
\end{align*}
where
\begin{align*}
\lambda_1(x) \coloneqq \lambda_{E + A \to EA} (x) ,\qquad \lambda_2(x) \coloneqq \frac{\lambda_{EA + B \to EAB}}{\lambda_{EA + B \to EAB} + \lambda_{EA \to E + A}}(x + EA - E - A)
\end{align*}
\begin{align*}
\lambda_3(x) \coloneqq \bigg(&\frac{\lambda_{EAB \to EA + B}}{\lambda_{EAB \to EA + B} + \lambda_{EAB \to E + P + Q}}(x - E - A - B + EAB) \\
&\times \frac{\lambda_{EA + B \to EAB}}{\lambda_{EA + B \to EAB} + \lambda_{EA \to E + A}}(x + EA - E - A)\bigg)^{-1},
\end{align*}
and
\begin{align*}
\lambda_4(x) \coloneqq &\frac{\lambda_{EAB \to E + P + Q}}{\lambda_{EAB \to EA + B} + \lambda_{EAB \to E + P + Q}}(x - E - A - B + EAB).
\end{align*}

\section{The conditional distribution as a stationary distribution}

\begin{condition}\label{con_2} If $\kR(\gamma)=(y+U_i,y'+U_i)$ with $\gamma\in \Xi_i$, for some $i=1,\ldots,m$,  then  i)  $y\not< y'$, and  ii)  $y\not> y'$.
\end{condition}

Condition \ref{con_1} is a stronger version of Condition \ref{con_2}.

\begin{lemma}\label{lmm_con1-2}
Let $\cN$ be   weakly reversible and $\cU$ a set of eliminable non-interacting species. Then, Condition \ref{con_1} implies Condition \ref{con_2}.
\end{lemma}

\begin{theorem}\label{thm_sdpi0}
Let $(\cN,\lambda)$ be a weakly reversible SRN with  a set $\cU$ of non-interacting species fulfilling Condition \ref{con_2}. Let $(\cN_{\cU},\lambda_{\cU})$ be the  reduced SRN of $(\cN,\lambda)$ by eliminating  the species in $\cU$. Let $\Gamma$ be an irreducible component of $\cN$ such that
\begin{align}\label{def_gmm0-nint}
\Gamma_0 \coloneqq \big\{ x\in \Gamma \colon \rho(x) = 0 \big\} \neq \emptyset.
\end{align}
Suppose that $\pi$ is a complex balanced distribution for $(\cN,\lambda)$ on $\Gamma$.  Then,
\begin{enumerate}

\item[i)] $\Gamma_0$ is either an irreducible component of $\cN_{\cU}$ or the union of at most countable many disjoint irreducible components of $\cN_{\cU}$. 

\item[ii)] The conditional distribution $\pi_0$ of $\pi$ restricted to $\Gamma_0$ is a stationary distribution for $(\cN_{\cU},\lambda_{\cU})$ on $\Gamma_0$, if $\Gamma_0$ consists of a single irreducible component, and otherwise,  it is a linear combination of stationary distributions.  
\end{enumerate}
\end{theorem}

For finite  state space $\Gamma$, the theorem follows from \cite{HW2021}.  
The proof is lengthy and technical, and thus, deferred to Section \ref{prf_sdpi0}.

Combining Lemmas \ref{lmm_fgrn}, \ref{lmm_con1-2} and Theorem \ref{thm_sdpi0}, the next proposition follows.

\begin{corollary}
If $\cN_{\cU}$ has finitely many reactions, then both conclusions in Theorem \ref{thm_sdpi0} are valid.
\end{corollary}

\subsection{Two special cases}

\begin{definition}
 A set of eliminable non-interacting species $\cU$ is a set of {\it intermediate species}, if $y = \rho(y)$ for every $y\in \cup_{i=1}^m\cC_i$. 
\end{definition}

It can be  demonstrated that both Conditions \ref{con_1} and \ref{con_2} are satisfied. Therefore, Theorem \ref{thm_sdpi0} applies.
For  a set of intermediate species $\cU$, it holds  that for  $i,j \in\{1,\dots, m\}$,  
\begin{enumerate}
 \item[i)]  $\cC_i = \{U_i\}$.  
\item[ii)] $\cR_{i,i} = \emptyset$ and $\cR_{i,j} \subseteq \{U_i\ce{->}U_j\}$,  $i\neq j$.  
\end{enumerate}

\begin{proposition}\label{prop_itms}
 Suppose   $\pi$ is a complex balanced distribution for  $(\cN,\lambda)$ on  an irreducible component $\Gamma$ and that $\cU$ consists of intermediate species. Then, $\Gamma_0\coloneqq \{x\in \Gamma\colon \rho (x) = 0\}$ is an irreducible component of $(\cN_{\cU},\lambda_{\cU})$, and $\pi_{0}(\cdot) \coloneqq  \pi(\cdot)/\pi(\Gamma_{0})$ is a complex balanced distribution for $(\cN_{\cU},\lambda_{\cU})$ on $\Gamma_0$.
\end{proposition}

\begin{proposition}\label{prop_rvsb}
 Suppose   $\pi$ is a detailed balanced distribution for $(\cN,\lambda)$ on  an irreducible component $\Gamma$. Then, $\pi_{0}(\cdot) \coloneqq \pi(\cdot)/\pi(\Gamma_{0})$ is either a detailed balanced distribution for $(\cN_{\cU},\lambda_{\cU})$ on $\Gamma_0$, as in \eqref{def_gmm0}, or a linear combination of detailed balanced distributions on disjoint irreducible components.
\end{proposition}

The proofs of Propositions \ref{prop_itms} and \ref{prop_rvsb} are in Sections \ref{prf_itms} and \ref{prf_rvsb}, respectively.

  \subsection{Positive recurrence of the   reduced SRN}

 \begin{theorem}\label{cons_pos_red}
Let $(\cN,\lambda)$ be an SRN with stochastic mass-action kinetics, $\cU$ a set of non-interacting species, $(\cN_{\cU},\lambda_{\cU})$ the   reduced SRN obtained by elimination of the species in $\cU$, and $\Gamma$ an irreducible component of $(\cN,\lambda)$.
 Let $\pi$ be a complex balanced stationary distribution of $(\cN,\lambda)$  on $\Gamma$, and assume $\pi_{0}(\cdot) \coloneqq \pi(\cdot)/\pi(\Gamma_{0})$ is a stationary distribution (linear combination of stationary distributions on disjoint irreducible components) 
 of $(\cN_{\cU},\lambda_{\cU})$ supported on $\Gamma_0\not=\emptyset$, as in \eqref{def_gmm0-nint}. Then, $(\cN_{\cU},\lambda_{\cU})$ is positive recurrent on each irreducible component of $\Gamma_0$.
\end{theorem}

\begin{proof} 
By assumption, $\pi_{0}$ is a stationary distribution on an  irreducible component $\tilde{\Gamma}\subseteq\Gamma_0$. To show positive recurrence on $\tilde{\Gamma}$, we apply the sufficient condition for positive recurrence from \cite[Corollary 4.4]{asmussen}.  The condition holds if for $\tilde{\Gamma}$ the following holds:
 $$\sum_{z\in\tilde{\Gamma}}\sum_{r\in\mR_{\cU}}\lambda_{\cU,r}(z)\pi_0(z)<\infty.$$ 
Using that complex balanced distributions for mass-action SRNs take  the form $M\frac{c^x}{x!}$ \cite{anderson2}, we can replace  $\pi_0(z)=\pi(z)/\pi(\Gamma_{0})$  by $M_0\frac{c^x}{x!}$ with $M_0=M/\pi(\Gamma_0)$ and $c^x \coloneqq \prod_{i=1}^n c_i^{x_i}$, for any $c \in \R_+^n$ and $x\in \R^n$. Furthermore,  by \cite[Lemma 3.10]{HW2021},
$$\sum_{r\in\mR_\cU}\lambda_{\cU,r}(z)\leq \sum_{r\in\mR}\lambda_{r}(z,0)$$
Since $\tilde{\Gamma}\subseteq\Gamma_0\subseteq \N^{n_\cU}_0$ and   $\mR$ is finite, it is enough to show that
$$M_0\sum_{z\in\Z^{n_\cU}_{\geq 0}}\lambda_{r}(z,0)\frac{c^{(z,0)}}{z!}<\infty. $$
This final inequality follows as $\lambda_{r}(z,0)$ is a polynomial of fixed degree  (in $z$) for stochastic mass-action kinetics.
\end{proof}

As a corollary of   Theorem \ref{thm_sdpi0}, Proposition \ref{prop_itms}, Proposition \ref{prop_rvsb}, combined with Theorem \ref{cons_pos_red}, we get the following for  reductions of complex balanced SRNs with stochastic mass-action kinetics. 

\begin{corollary}
Let $(\cN,\lambda)$ be an SRN with stochastic mass-action kinetics, $\cU$ a set of non-interacting species, $(\cN_{\cU},\lambda_{\cU})$ the   reduced SRN obtained by elimination of the species in $\cU$, and $\Gamma$ an irreducible component of $(\cN,\lambda)$.  Let $\pi$ be a complex balanced stationary distribution of $(\cN,\lambda)$ on $\Gamma$.  
Assume that at least one of the following conditions hold: 
\begin{itemize}
\item Condition \ref{con_2},
\item $\cU$ consists of intermediate species,
\item $\cN$ is reversible.
\end{itemize}
Then, $(\cN_{\cU},\lambda_{\cU})$ is positive recurrent on each irreducible component of $\Gamma_0$.
\end{corollary}
 
\subsection{A remark on Theorem \ref{thm_sdpi0}}
  
The proof of Theorem \ref{thm_sdpi0}ii) presented in Section \ref{prf_sdpi0} is based on the recurrence of two irreducible discrete-time Markov chains (DTMCs). Recurrence is deduced from the finiteness of the state spaces (of the DTMCs), as a result of Condition \ref{con_2}. It is worth noticing that  finiteness of the state space is not a necessary condition for  recurrence. However, at present, we have not been able to identify any other simple condition that  imply  recurrence of these DTMCs. The  non-necessity of Condition \ref{con_2} is demonstrated in the following example, although Theorem \ref{thm_sdpi0}ii) remains valid in this particular case.

\begin{example}
Consider the next reaction network
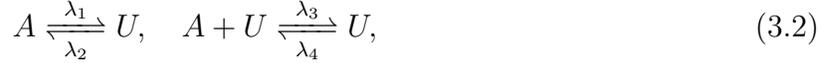
\begin{align}\label{exp_count}
A\ce{<=>[\lambda_1][\lambda_2]}U,\quad A+U\ce{<=>[\lambda_3][\lambda_4]}U,
\end{align}
where for all $x=(x_A,x_U)\in \N_0^2$,
\begin{equation*}
\begin{aligned}
\lambda_1(x)=&(x_A!)^2\1_{\N}(x_A),\\
\lambda_2(x)=&[(x_A+1)!]^2\1_{\N}(x_U),
\end{aligned}\qquad
\begin{aligned}
\lambda_3(x)=&2[(x_A+1)!]^2\1_{\N^2}(x_A,x_U),\\
 \lambda_4(x)=&[(x_A+2)!]^2\1_{\N}(x_U).
\end{aligned}
\end{equation*}
Then, $\Gamma=\N_0^2\setminus \{(0,0)\}$ is an irreducible component of \eqref{exp_count}. One can show that
\[
\pi(x)\coloneqq \frac{1}{3\times 2^{x_A+x_U}},\quad   x = (x_A,x_U)\in \Gamma,
\] 
is a detailed balanced distribution for \eqref{exp_count} on $\Gamma$.

 Following the same strategy as in  Step 1 of the proof of Theorem \ref{thm_sdpi0}; see in Section \ref{prf_sdpi0}, we construct a DTMC $X$ on $\cX = \{(x_A, 1)\colon x_A \geq 0\}\cup \{\partial\}$ (where $\partial$ is an additional state) with transition probability
\begin{align*}
P_{\partial,\partial} = 0,\quad P_{\partial, (x - 1, 1)} = 1,\quad P_{(x_A,1),\partial} = \frac{\lambda_2(x_A,1)}{(\lambda_2+\lambda_3+\lambda_4)(x_A,1)},
\end{align*}
\begin{align*}
P_{(x_A,1),(x_A + 1,1)} = \frac{\lambda_4(x_A,1)}{(\lambda_2+\lambda_3+\lambda_4)(x_A,1)},\quad P_{(x_A,1), (x_A - 1,1)} = \frac{\lambda_3(x_A,1)}{(\lambda_2+\lambda_3+\lambda_4)(x_A,1)},
\end{align*}
and otherwise $P_{z,z'} = 0$. Then, Theorem \ref{thm_sdpi0}ii) is valid if $\mathbb{P}_{(x-1,1)} (\tau < \infty) = 1$, where $\tau \coloneqq \min\{ n\geq 1\colon X_n = \partial\}$. However, this is unfortunately not true. In fact, for any $n\geq 2$,
\begin{align*}
\mathbb{P}_{(x-1,1)}(\tau > n) \geq &\mathbb{P}_{(x - 1,1)}\big(X_k=(x + k-1, 1),\ \forall k=1,\dots,n\big)\\
=& \prod_{k=1}^{n} P_{(x + k -2,1),(x + k - 1, 1)} = \prod_{k = 1}^{n} \frac{\lambda_4(x + k-2,1)}{(\lambda_2+\lambda_3+\lambda_4)(x + k - 2,1)}\\
=&\prod_{k=1}^{n}\frac{\prod_{j=1}^{x + k}j^2}{\big(\prod_{j=1}^{x + k - 1}j^2 \big) + \big(2\prod_{j=1}^{x + k - 1}j^2 \1_{\N}(x + k - 2) \big) + \big(\prod_{j=1}^{x + k} j^2\big) }\\
\geq &\prod_{k=1}^{n}\big(3 (x + k)^{-2}+1\big)^{-1}.
\end{align*}
By using the Taylor expansion for the logarithmic function, one can show that
\[
\log \Big(\prod_{k=1}^{\infty}\big(3 (x + k)^{-2}+1\big)^{-1}\Big) = - \sum_{k=1}^{\infty}\log \big(3(x + k)^{-2} + 1\big) \geq -\sum_{k=1}^{\infty} 3(x+k)^{-2} > -\infty.
\]
This implies $\prod_{k=1}^{\infty}\big(3(x + k )^{-2}+1\big)^{-1}>0$. Therefore, $\mathbb{P}_{(x - 1,0)}(\tau = \infty)>0$. 

  This implies that simply replicating the arguments in the proof of Theorem \ref{thm_sdpi0} does not suffice to show that $\pi_0$ is a stationary distribution for the   reduced SRN. This is due to the failure of fulfilling Condition \ref{con_2} in \eqref{exp_count}. However, applying Proposition \ref{prop_rvsb},   the reduce SRN  is not only stationary but also detailed balanced with $\pi_0$. We conjecture that Condition \ref{con_2} can be removed in Theorem \ref{thm_sdpi0} but are unable to prove it. 
\end{example}

\section{Convergence of complex balanced distributions} 

Consider an SRN $(\cN,\lambda)$ with $\cN=(\cS,\cC,\cR)$, and suppose $\cU$ is a proper subset of $\cS$  with $|\cU|=m$. 
  In particular,  $\cU$ does not need to be a set of non-interacting species.

Let  $\beta=(\beta_1,\dots, \beta_m)\in\R^m_{>0}$. For  $N\in\N$, let $\lambda_{N}^{\beta}=(\lambda_{N,y\to y'}^{\beta}\colon y\ce{->}y'\in \cR)$ be a  kinetics on $\cN$ defined by
\begin{align}\label{def_lmdn}
\lambda_{N, y\to y'}^{\beta}=N^{\langle \beta, \rho (y)\rangle}\lambda_{y\to y'},\quad  y\ce{->}y'\in\cR,
\end{align}
where $\rho$ is the projection from $\N_0^n \to \N_0^m$ defined   in \eqref{def_rho}.

\begin{proposition}\label{thm_sdsc}
 Suppose that $\pi$ is a complex balanced distribution for $(\cN,\lambda)$ on  an irreducible component $\Gamma$. 
Define $g_{N}^{\beta}(x)=N^{-\langle\beta, \rho (x) \rangle}\pi(x)$ for $N\in \N$.
Then, $M_{N}^{\beta}=\sum_{x\in \Gamma}g_{N}^{\beta}(x)\in (0,1)$ and $\pi_{N}^{\beta}(x)=\frac{1}{M_{N}^{\beta}}g_{N}^{\beta}(x)$ is a complex balanced distribution for $(\cN,\lambda_N^\beta)$ on $\Gamma$.
\end{proposition}

\begin{proof}
Since $\pi$ is a complex balanced distribution for $(\cN,\lambda)$ on $\Gamma$, then \eqref{def_cbd} states that for any $\eta\in \cC$ and $x\in\Gamma$,
\begin{align*}
\pi(x)\sum_{y'\colon \eta\to y'\in \cR}\lambda_{\eta\to y'}(x)=\sum_{y\colon y\to \eta\in \cR}\pi(x+y-\eta)\lambda_{y\to \eta}(x+y-\eta).
\end{align*}
By definition of $\lambda_N^{\beta}$ and $g_N^{\beta}$, and by  linearity of $\rho $, we have
\begin{align*}
g_N^{\beta}(x)\sum_{y'\colon \eta\to y'\in \cR}&\lambda_{N,\eta\to y'}^{\beta}(x)=N^{-\langle \beta,\rho (x-\eta)\rangle}\pi(x)\sum_{y'\colon \eta\to y'\in \cR}\lambda_{\eta\to y'}(x)\\
& =N^{-\langle \beta,\rho (x-\eta)\rangle} \sum_{y\colon y\to \eta\in \cR}\pi(x+y-\eta)\lambda_{y\to \eta}(x+y-\eta)\\
& =\sum_{y\colon y\to \eta\in \cR} N^{-\langle \beta, \rho (x+y-\eta) \rangle} \pi(x+y-\eta)\  N^{\langle \beta, \rho (y) \rangle} \lambda_{y\to \eta}(x+y-\eta) \\
 & =\sum_{y\colon y\to \eta\in \cR}g_N^{\beta}(x+y-\eta)\lambda_{N,y\to \eta}(x+y-\eta),
\end{align*}
provided that $x\geq \eta$. Furthermore, for $x\not\geq \eta$, 
\[
0 = g_N^{\beta}(x) \sum_{y'\colon \eta\to y'\in \cR} \lambda_{N,\eta\to y'}^{\beta}(x) = \sum_{y\colon y\to \eta\in \cR} g_N^{\beta}(x+y-\eta) \lambda_{N,y\to \eta}^{\beta}(x+y-\eta).
\]
Note that $g_N^{\beta} (x) \leq \pi(x)$ for all $x\in \Gamma$, and $g_N^{\beta}(x)>0$ if and only if $\pi(x)>0$. This implies that $M_{N}^{\beta}=\sum_{x\in \Gamma}g_{N}^{\beta}(x)\in (0,1)$ and $\pi_N$ is a complex balanced distribution for $(\cN,\lambda_N^{\beta})$ on $\Gamma$. The proof is complete.
\end{proof}

We have the following theorem of the asymptotic behaviour of $\pi_{N}^{\beta}$. 

\begin{theorem}\label{cbal_red}
 Assume the conditions in Proposition \ref{thm_sdsc}. Let 
 \begin{align}\label{def_gmm0}
 \gamma_0^{\beta} = \min\{\langle\beta, \rho (x)\rangle\colon x\in \Gamma\},\quad\text{and}\quad\Gamma_0^{\beta}=\{x\in \Gamma\colon \langle\beta, \rho (x)\rangle=\gamma_0^{\beta}\}.
 \end{align}
Then, for every $x\in\Gamma$,
\begin{align}\label{lmtsd}
\pi_0^{\beta}(x) = \lim_{N\to\infty} \pi_{N}(x) &= \begin{cases}\pi(x)/\pi(\Gamma_0^{\beta}), &{\rm if}\ x\in \Gamma_0^{\beta}, \\
0, & {\rm if}\ x\in \Gamma\setminus\Gamma_0^{\beta}.
\end{cases}
\end{align}
\end{theorem}
\begin{proof}
Due to Proposition \ref{thm_sdsc}, we have
\[
\pi_{N}(x)=\frac{N^{-\langle\beta, \rho (x) \rangle} \pi(x)}{\sum_{x'\in \Gamma}N^{-\langle\beta, \rho (x')\rangle} \pi(x')}.
\]
Decompose $\Gamma=\Gamma_0^{\beta}\cup (\Gamma\setminus \Gamma_0^{\beta})$. Then it follows that
\[
\sum_{x'\in \Gamma}N^{-\langle\beta, \rho (x')\rangle} \pi(x')=\sum_{x'\in \Gamma_0^{\beta}}N^{-\gamma_0^{\beta}} \pi(x')+\sum_{x'\in \Gamma\setminus \Gamma_0^{\beta}}N^{-\langle\beta, \rho (x')\rangle} \pi(x'),
\]
and thus
\begin{align*}
\pi_{N}(x)=&\frac{N^{-\langle\beta, \rho (x) \rangle} \pi(x)}{\sum_{x'\in \Gamma_0^{\beta}}N^{-\gamma_0^{\beta}} \pi(x')+\sum_{x'\in \Gamma\setminus \Gamma_0^{\beta}}N^{-\langle\beta, \rho (x')\rangle} \pi(x')}\\
&=\frac{N^{\gamma_0^{\beta}-\langle \beta, \rho (x)\rangle } \pi(x)}{\sum_{x'\in \Gamma_0^{\beta}} \pi(x')+\sum_{x'\in \Gamma\setminus \Gamma_0^{\beta}}N^{\gamma_0^{\beta}-\langle\beta, \rho (x')\rangle} \pi(x')}.
\end{align*}
Note that $\Gamma$ is an irreducible component of $\cN$, and $\pi$ is a complex balanced distribution for $(\cN,\lambda)$. Thus, $\pi(x)>0$ for all $x\in \Gamma$. Therefore, for $N\to\infty$,   \eqref{lmtsd} follows from the fact that on the set $\Gamma$, $\gamma_0^{\beta}-\langle\beta, \rho (x')\rangle>0$ if and only if $x'\in \Gamma\setminus \Gamma_0^{\beta}$. The proof of this theorem is complete.
\end{proof}

 The distribution  $\pi_0^{\beta}$ in Theorem \ref{cbal_red} is the conditional probability of $\pi$ on $\Gamma_0^{\beta}$. By replacing $\pi$ with $\pi_N$ (defined in Proposition \ref{thm_sdsc}) for any $N$, we can conclude that $\pi_0^{\beta}$ is the conditional probability of $\pi_N^{\beta}$ on $\Gamma_0$. This follows from  $\pi_N^{\beta} (x) = N^{-\gamma_0^{\beta}} \pi(x)$ for all $x\in\Gamma_0^{\beta}$.
In comparison with \ref{exp_1}, Proposition \ref{thm_sdsc} and Theorem \ref{cbal_red} show that the mass-action assumption is not necessary, and can be replaced by complex balancing.

As a result of Proposition \ref{thm_sdsc}, Theorem \ref{thm_sdpi0}/Proposition \ref{prop_itms}/ Proposition \ref{prop_rvsb} and Theorem \ref{cbal_red}, we have the next proposition.

\begin{theorem}
Let $(\cN,\lambda)$ be a weakly reversible SRN with a set of non-interacting species $\cU$, and let $(\cN, \lambda_N^{\beta})$ be the   SRN with $\lambda_N^{\beta}$,  $N\in \N$; and $\beta \in \R_{> 0 }^{m}$, defined   in \eqref{def_lmdn}. Suppose that $\pi$ is a complex balanced distribution for $(\cN,\lambda)$ on an irreducible component $\Gamma$, and that $\Gamma_0$ defined as in \eqref{def_gmm0-nint} is non-empty. Assume that at least one  of the following conditions hold:
\begin{itemize}
\item Condition \ref{con_2},
\item $\cU$ are intermediates,
\item $\cN$ is reversible,
\end{itemize}
and let   $(\cN_{\cU},\lambda_{\cU})$ be the  reduced SRN of $(\cN,\lambda)$ by elimination of the species in $\cU$.
Then, $\pi_0^{\beta} = \lim_{N\to\infty}\pi_N^{\beta}$, where $\pi_0$ is defined in Proposition \ref{thm_sdsc}, is a stationary distribution (linear combination of stationary distributions on disjoint irreducible components of $\cN_{\cU}$) of $(\cN_{\cU},\lambda_{\cU})$ supported on $\Gamma_0$. 
\end{theorem}

\section{Proofs}

In this section, we give  proofs of  Lemma \ref{lmm_fgrn}, Theorem  \ref{thm_sdpi0}, Propositions \ref{prop_itms} and \ref{prop_rvsb}. In particular, the proofs of Theorem  \ref{thm_sdpi0} and Proposition \ref{prop_itms} extend  ideas  in \cite[Lemmas 4.2 and 4.3]{HWX21+}.

We introduce some notation. Let $ \cN $ be a (reduced) reaction network.
 A reaction $r=y\ce{->}y$ is   an \emph{incoming} reaction of $y'$ and an \emph{outgoing} reaction of $y$.   For a complex, $y\in \cC$, $\inc (y)$ and $\out (y)$  denote the sets of all incoming and outgoing reactions of $y$, respectively.  In a reduced reaction network, we often add the subscript $\cU$, for example, $\out_\cU$.

For $\gamma\in\Xi_i(q)$, let 
\begin{align*}
\init(\gamma)=U_i\ce{->[$r_0$]}U_{i_1}\quad \mathrm{and}\quad \ter(\gamma) = U_{i_q}\ce{->[$r_q$]}U_{i},
\end{align*}
be the initial and terminal edges of $\gamma$, respectively. Finally, we define 
\begin{align}\label{def_mpro}
m_{\pro} \coloneqq |\cU_{\pro}|,\quad \text{and} \quad \cU_{\pro} = \{U_1,\dots, U_{m_{\pro}}\},
\end{align}
such that $\cU\setminus \cU_{\pro} =\{U_{m_{\pro}+1},\ldots,U_m\}$.

\subsection{Proof of Lemma \ref{lmm_fgrn}}\label{app_fgrn}

($\Rightarrow$) \quad  
 Suppose Condition \ref{con_1} holds. Given a walk $\theta\subseteq (\cV, \cE)$, define $\cU_{\theta} \coloneqq \{U_i\in \cU \colon U_i \in \theta\}$.  Consider a closed walk $\gamma \in \Xi_{i_0}$ for any $i_0 \in \{1,\dots, m_{\pro}\}$. Then, $U_{i_0} \in \cU_{\gamma}$. Additionally, $\cR_{i_0,i_0}=\emptyset$, because otherwise $r = y + U_{i_0} \ce{->} y' + U_{i_0} \in \cR_{i_0,i_0}$ with $y \neq y'$ contradicts   Condition \ref{con_1}. This implies $|\cU_{\gamma}|\geq 2$, and  
\[
\kR(\Xi_{i_0})=\cup_{k=2}^{m_{\pro}}A_{i_0,k}, \quad A_{i_0, k} \coloneqq \{\kR(\gamma)\colon |\cU_{\gamma}|=k\}.
\]
 We will show that $|\kR(\Xi_0)| < \infty$, which is the key to   prove $|\cR_{\cU}| < \infty$.

Suppose $k=2$. Then, $\gamma$ has the form
\[
\gamma = U_{i_0}\ce{->[$r_0$]}U_{j}\ce{->[$r_1$]}U_{i_0}\ce{->[$r_{2}$]}\cdots\ce{->[$r_{2q}$]}U_j\ce{->[$r_{2q+1}$]}U_{i_0},
\]
for some $j \in \{1,\dots, i_0 - 1, i_0 + 1,\dots, m_{\pro}\}$ and $q\in\N_0$. As a result of \cite[Proposition 1]{hoessly2021sum} and Condition \ref{con_1}, 
\begin{align}\label{eq_fgrn1}
\kR(\gamma) = \oplus_{i=0}^q \big(r_{2i}\oplus r_{2i+1}\big) = \oplus_{i=0}^q (y_i+U_{i_0},y_i+U_{i_0}),
\end{align}
for some $y_i \in \N_0^{\cS}$, $i=0,\dots, q$. Note that $\{r\oplus r' \colon r \in \cR_{i_0,j}, r' \in \cR_{j,i_0}\}$ is a finite set. This implies that 
\begin{align*}
  \big\{ y \in \Z_{\geq 0}^{\cS} \colon (y + U_{i_0}, y + U_{i_0}) = r \oplus r',  r \in \cR_{i_0,j}, r' \in \cR_{j,i_0}\big\}
\end{align*}
is a finite set.
  Combining this observation with \cite[Proposition 2(iii)]{hoessly2021sum}, we conclude that $A_{i_0,2}$ is   finite. 

We prove   finiteness of $A_{i_0,k}$ for general $k\geq 3$ by induction. Let $\gamma \in A_{i_0,k}$ be of the form \eqref{def_gm} with $i=i_0$.  It suffices to prove the claim
\begin{equation}\label{eq_claim}\tag{$*$}
   \text{The collection of all possible $\kR(\gamma)$ is finite, assuming $i_0 \notin \{i_1,\dots, i_q \}$}
 \end{equation}  
 Otherwise, we can decompose $\gamma$ into a `summation' of closed walks,
where $U_{i_0}$ only appears as the initial and terminal nodes in each closed walk.  Here, we say a walk $\theta$ is the summation of two walks $\theta_1, \theta_2$, written as $\theta = \theta_1 + \theta_2$ if 
\[
  \theta = U_{i_1}\ce{->[$r_1$]}U_{i_2}\ce{->[$r_2$]}\cdots \ce{->[$r_{q''-1}$]}U_{i_{q''}},
\]
\[
\theta_1 = \{U_{i_1}\ce{->[$r_1$]}\cdots \ce{->[$r_{q'-1}$]} U_{i_{q'}}\}\quad \mathrm{and}\quad \theta_2 =
\{U_{i_{q'}}\ce{->[$r_{q'}$]}\cdots \ce{->[$r_{q''-1}$]}U_{i_{q''}}\}.
\]
Therefore, $\kR(\gamma)$ can be written in the form of \eqref{eq_fgrn1},  
 and finiteness of $A_{i_0,k}$ follows from \cite[Proposition 2(iii)]{hoessly2021sum} and  the finiteness of the collection of all possible $y_i$'s is finite, where the latter is a consequence of claim \eqref{eq_claim}.

Suppose that $\gamma$ of the form \eqref{def_gm} with $i=i_0$ and $i_0 \notin \{i_1,\dots, i_q\}$. Then the collection of all possible $\gamma$'s can be divided into two sets:
\begin{equation*}
  B_1 \coloneqq \{\gamma \colon i_{j}\neq i_{j'}, \text{ for all } 1\leq j < j' \leq q\},
\quad
  B_2 \coloneqq \{\gamma \colon \gamma \text{ is not in } B_1\}.
\end{equation*}
It is trivial that $B_1$ is a finite set. It suffices to show that $B_2$ is finite, too. Let $\gamma \in B_2$.
We   decompose $\gamma$ into the summation of closed walks connected by paths  
\[
\gamma = \theta_{1} + \gamma_{1} + \theta_{2} + \dots + \gamma_{q} + \theta_{q+1},
\]
where for all $i$, $\gamma_{i}\in \Xi_{j_i'}$ and $\theta_{i}$ is a path in the sense that  no repeated nodes are allowed.

Without loss of generality, we assume that $j_i'\neq j_{i'}'$ for all $1\leq i< i'\leq q'$, where $j_i'$ is such that $\gamma_i \in \Xi_{j_i'}$. Thus, $q\leq m_{\pro}$. 
Because otherwise,  let $\gamma'_{i} = \gamma_{i} + \theta_{i+1} + \cdots + \theta_{i'} + \gamma_{i'}\in \Xi_{j_i'}$,  and write 
\[
\gamma = \theta_{1} + \gamma_{1} + \theta_{2} + \dots + \theta_{i} + \gamma_i' + \theta_{i' + 1} + \gamma_{i' + 1} + \dots + \gamma_{q} + \theta_{q+1}.
\]
 This process can be iterated until no repeated $j_i'$'s occur.
Note that the collection of all paths in $(\cV,\cE)$ is finite. This implies the collection of all possible $\kR(\theta_{i})$'s, $1\leq i\leq q'+1$ is finite. 
Additionally, for any $i\in\{1,\dots, q\}$, $U_{i_0}$ is not included in $\gamma_{i}$, and thus  $\kR(\gamma_{i})\in A_{j_i', k'}$ with some $k'<k$. Thanks to the induction hypothesis, $|A_{j_i', k'}| < \infty$,   the collection of all possible $\kR(\gamma_i)$ is finite for all $i = 1,\dots, q$. Hence,  $B_2$ (the collection of all possible $\kR(\gamma) = \kR (\theta_{1})\oplus \kR(\gamma_1) \oplus \cdots \oplus \kR(\theta_{q+1})$) is finite   due to \cite[Proposition 2(iii)]{hoessly2021sum}, and the fact $q\leq m_{\pro}$.  
  This completes the proof of claim \eqref{eq_claim} for $i_0$ and thus for all $i = 1,\dots, m_{\pro}$ and $k = 2, \dots, m_{\pro}$.

The rest of the proof is straightforward. For any $\gamma \in \Xi_0$, we can decompose it to a  
summation of at most $m_{\pro}$ closed walks connected by paths. Since  the range of the function $\kR$ on paths and closed walks are all finite,  the collections of all possible $\kR(\gamma)$  is also finite. As a result, $|\kR(\Xi_0)| < \infty$.

($\Leftarrow$)\quad   Suppose that $\cR_{\cU}$ consists of finitely many reactions but Condition \ref{con_1} fails. Then, there exists a closed walk $\gamma_0 \in \Xi_{i_0}$ for some $i_0\in \{1,\dots, m_{\pro}\}$, such that $\kR(\gamma_0) = y_0 + U_{i_0} \ce{->} y_0' + U_{i_0}$ with $y_0\neq y_0'$. Since $\cU_{\pro}\subseteq\cU_{\deg}$ (the set $\cU$ can be eliminated), then there are  paths $\theta_1$ and $\theta_2$ directed from $U_0$ to $U_{i_0}$ and  from $U_{i_0}$ to $U_{0}$, respectively. It follows that for any $q\in \N$,  
\[
\gamma_{q}\coloneqq \theta_1 + \gamma_{q}^* +\theta_2 \in \Xi_0 \quad \mathrm{where}\quad \gamma_{q}^* \coloneqq \underbrace{\gamma + \dots + \gamma}_{\text{$q$ instances of $\gamma$}}.
\]
Let $(x_q,x_q') \coloneqq \kR(\gamma_{q})$.
We will show that  $\{ (x_q,x_q') \colon q\in \N\}\cap \cR_\cU$
is an infinite set.  Recall that $y_0\neq y_0'$.
 Without loss of generality, assume that $y_{0,1} \neq y_{0,1}'$. Here, we only prove the case $y_{0,1}>y_{0,1}'$, while the other case $y_{0,1} < y_{0,1}'$ can be showed similarly. Let
\[
(z_q+U_{i_0},z_q'+U_{i_0}) \coloneqq \kR(\gamma_{q}^*)=\oplus_{i=1}^q(y+U_{i_0},y'+U_{i_0}).
\]
Then, by Definition \ref{def_sum} and  since $y_{0,1}>y_{0,1}'$, we have 
\[
(z_{q,1},z_{q,1}')=(y_{0,1}+q(y_{0,1} - y_{0,1}'),y_{0,1}').
\]
This yields that $\{z_{q,1}\colon q \in \N\}$ is a strictly increasing sequence. Let $\kR(\theta_1)=(y_1,y_1'+U_{1})$ and $\kR(\theta_2)=(y_2+U_{1},y_2')$. Then, for all $q > M$ for $M$ large enough, it holds that $z_{q,1}\geq y_{1,1}'$. As a result,
\begin{align*}
x_{q,1} &=y_{1,1}+z_{q,1}-y_{0,1}'+(y_{2,1}-z_{q,1}')\vee 0\\
 &=y_{1,1}+y_{0,1}+q(y_{0,1}-y_{0,1}')-y_{1,1}'+(y_{2,1}-y_{0,1}')\vee 0.
\end{align*}
Thus, $\{x_{q,1}, q > M\}$ is a strictly increasing sequence. This  contradicts that $\{\kR(\gamma_{q}) = (x_{q},x_{q}') \colon q>M\} \subseteq \cR_\cU$ is a {finite} set. The proof is complete.

\subsection{Proof of Theorem \ref{thm_sdpi0}}\label{prf_sdpi0}

\begin{proof}[Proof of Theorem \ref{thm_sdpi0}i)]
It suffice to show that for any $x\in \Gamma_0$ and $x'\in\N_0^{n_{\cU}}$, $x\to_{\cN_{\cU}} x'$ implies 1) $x'\in\Gamma_0$, and 2) $x'\to_{\cN_{\cU}} x$. 

By definition, $x\to_{\cN_{\cU}} x'$ if and only if there exist reactions $r_1,\dots, r_k \in \cR_{\cU}$, such that $r_1\oplus \cdots \oplus r_k = (y,y')$ with $x\geq y$ and $x' - x = y' - y$. Using an iteration argument, we can assume that $k = 1$. By definition of reduction of reaction networks,  either $r_1 \in \cR$ or there exists a $\gamma \in \Xi_0$ such that $r_1 = \kR (\gamma)$. In either case, we have $x\to_{\cN} x'$ with $\rho(x') = 0$ and thus $x' \in \Gamma_0$.

By the weak reversibility of $\cN$, if $r_1 = y_1\ce{->}y' \in \cR$, then there exists a sequence of reactions such that $y'\ce{->}y_1\ce{->}y_2\ce{->}\cdots \ce{->}y_{k'}\ce{->} y \subseteq \cR$.   We further assume that $y_i \neq y_j$ for all $1 \leq i < j \leq k'$. Otherwise if $y_i = y_j$, one can work on the next sequence $y'\ce{->}y_1\ce{->} \cdots \ce{->}y_{i} \ce{->}y_{j + 1} \ce{->} \cdots \ce{->} y$, which is also included in $\cR$. If $y_i \in \cC_{0}$ for all $i=1,\dots, k'$, then each reaction is also an element of $\cR_{\cU}$. This proves $x'\to_{\cN_{\cU}} x$. Otherwise, there are $1\leq i_1<i_2<\dots< i_{k''}\leq k'$ for some $k''<k'$, such that $y_{i}\in \cC_0$ if and only if $i \in \{i_1,\dots, i_{k''}\}$. For each $q \in \{0,\dots, k''\}$, let $\gamma_q$ be the closed walk in $(\cV,\cE)$ given by the reactions 
$y_{i_q} \ce{->} y_{i_q+1}   \ce{->}\cdots \ce{->} y_{i_{q+1}}$, where by convention $y_{i_0}\coloneqq y'$ and $y_{i_{k'+1}} \coloneqq y$.
Then, $\gamma_q \in \Xi_0$ for all $q$. By removing all $\gamma_q$, such that $\kR(\gamma_q) = (y,y)$ for some $y\in\N_0^n$. 
Recall the assumption that that $y_i \neq y_j$ for all $1 \leq i < j \leq k'$, we know that $r_{(q)} \coloneqq  \kR(\gamma_q) = y_{i_q} \ce{->} y_{i_{q + 1}} \in \cR_{\cU}$.
Therefore, $\oplus_{q=0}^{k''} r_{(q)} = (y',y)$. This proves $x'\to_{\cN_{\cU}} x$. The proof of Theorem \ref{thm_sdpi0}i) is complete.
\end{proof}

\begin{proof}[Proof of Theorem \ref{thm_sdpi0}ii)]   For any $\eta \in \cC_0$, let 
\[
L_0 (\eta) = \sum_{r \in \out_{\cU}^* (\eta)} \pi(x) \lambda_{\cU,r}(x) = \sum_{r\in \out (\eta)\cap \cR_{0,0}} \pi(x)\lambda_r(x)+\sum_{\gamma\in \Xi_0\colon \init(\gamma)\in \out(\eta)} \pi(x)\lambda^*_{\gamma}(x),
\]
\[
R_0 (\eta) = \sum_{r\in \inc (\eta)\cap \cR_{0,0}} \pi(x-\zeta_r)\lambda_r(x-\zeta_r)+\sum_{\gamma\in \Xi_0 \colon \ter(\gamma)\in \inc(\eta)} \pi(x-\zeta_{\kR(\gamma)})\lambda^*_{\gamma}(x-\zeta_{\kR(\gamma)}),
\]
where $x\in \Gamma_0$ satisfies $x\geq \eta$. We use $L (\eta)$ and $R (\eta)$ for the left- and right-hand sides, respectively, of the following complex balanced equation,
\begin{align*}
\pi(x)\sum_{r\in \out (\eta)} \lambda_r(x)=\sum_{r\in \inc (\eta)} \pi(x-\zeta_r)\lambda_r(x-\zeta_r).
\end{align*}
Then, we   prove $L_0 (\eta) = R_0 (\eta)$ by verifying $L (\eta) = L_0(\eta)$ and $R (\eta)=R_0 (\eta)$. Once this has been done, because $L_0(\eta) = R_0(\eta) = 0$ if $x\not\geq \eta$, we have
\[
\pi(x)\sum_{r\in \kR(\Xi_0)} \lambda_{\cU,r}(x) = \sum_{\eta \in \cC_0} L_0(\eta) = \sum_{\eta \in \cC_0} R_0(\eta) =\sum_{r\in \kR(\Xi_0)} \pi(x-\zeta_r)\lambda_{\cU,r}(x-\zeta_r).
\]
This implies that
\[
\pi(x)\sum_{r\in \cR_{\cU}} \lambda_{\cU,r}(x)=\sum_{r\in \cR_{\cU}} \pi(x-\zeta_r)\lambda_{\cU,r} (x-\zeta_r),
\]
and thus $\pi$ is a stationary distribution for $(\cN_{\cU},\lambda_{\cU})$. 

To achieve this goal, we introduce two DTMCs, which are  used in the proofs of $L (\eta) = L_0(\eta)$ and $R (\eta)=R_0 (\eta)$. Denote by $\Theta_0$ the set of all open (not closed) walks in $(\cV,\cE)$ of the form
\[
\theta=\{U_0\ce{->[$r_0$]}U_{i_1}\ce{->[$r_1$]}\cdots \ce{->[$r_{q-1}$]}U_{i_q}\},\quad q\in \N_0,\quad \mathrm{and}\quad i_1,\dots, i_q \in \{1,\dots, m_{\pro}\}.
\]
Define $\cX \coloneqq \cX_0\cup \{\partial\}$, where $\partial$ is a ceremony  state, and
\begin{align}\label{def_cx}
\cX_0 \coloneqq \{x+\zeta_{\kR(\theta)}\colon \theta \in \Theta_0,x\geq \reac(\kR(\theta))\}.
\end{align}
Analogously,   define $\cY \coloneqq \cY_0 \cup \{\partial\}$, where  
\begin{align}\label{def_cy}
\cY_0 \coloneqq \{y\in \N_0^n \colon y+\zeta_{\kR(\theta)} = x, \theta \in \Theta_0', y\geq \reac(\kR(\theta))\},
\end{align}
with $\Theta_0'$ consisting of all open walks in $(\cV,\cE)$ of the form
\[
\theta=\{U_{i_1}\ce{->[$r_1$]}\cdots \ce{->[$r_{q-1}$]}U_{i_q}\ce{->[$r_q$]} U_0\},\quad q\in \N_0,\quad \mathrm{and}\quad i_1,\dots, i_q \in  \{1,\dots, m_{\pro}\}.
\]

\begin{lemma}\label{lmm_fnit2}
i) $\cX_0$ in \eqref{def_cx} is finite for any $x\in \N_0^n$, if and only if Condition \ref{con_2}i) holds.

ii) $\cY_0$ in \eqref{def_cy} is finite for any $x\in \N_0^n$, if and only if Condition \ref{con_2}ii) holds.

\end{lemma}
The proof of Lemma \ref{lmm_fnit2} is in Section \ref{app_fnit2}.

\textbf{Step 1:}  $L(\eta) = L_0(\eta)$.
Similarly to Step 1 in Section \ref{prf_itms}, to prove $L(\eta) = L_0(\eta)$, it suffices to show that for every reaction $r_* = \eta\ce{->}y'+U_{\iota} \in \cR_{0,\iota}$,
\begin{align}\label{equ_sdl}
\sum_{k=1}^{\infty}\sum_{\gamma\in \Xi_0(k) \colon \init(\gamma) = r_*} \lambda^*_{\gamma}(x)= \lambda_{r_*}(x).
\end{align}

To prove \eqref{equ_sdl}, we introduce an auxiliary  discrete-time Markov chain (DTMC) $X$, and relate \eqref{equ_sdl} to the recurrence of $X$. For a clearer understanding, we refer readers to Section \ref{prf_itms}, where a simpler instance is detailed to illustrate how the recurrence of $X$ leads to the validity of \eqref{equ_sdl}. Due to the non-interacting structure,   for any $z\in \cX_0$, we have $\rho(z) = U_i$ for some $U_i\in \cU_{\pro}$. This allows us to define   the  transition probabilities  of the Markov chain $X$  as follows,
\begin{align*}
\mathbb{P}_{z} (z') = P_{z,z'} \coloneqq \begin{cases}
\displaystyle  \sum_{r\in \cR_{0,0}}\beta_{r}(x),  & z = z' = \partial,\\[12pt]
\displaystyle  \sum_{r \in \cR_{i,0}} \beta_{r}(z), & \rho (z) = U_i, z' = \partial,\\[12pt]
\displaystyle  \sum_{r \in \cR_{0,j}\colon \zeta_r = z' - x }\beta_{r}(x),  & z = \partial, \rho (z') = U_j\\[12pt]
\displaystyle \sum_{r\in \cR_{i,j}\colon  \zeta_r = z' - z} \beta_{r} (z), & \rho (z) = U_i, \rho (z') = U_j,
\end{cases}
\end{align*}
where $i,j\in \{1,\dots,m_{\pro}\}$. Moreover, \eqref{equ_sdl} is valid if and only if
\begin{align}\label{equ_sd2}
\mathbb{P}_{x+\zeta_{r_0}}(\tau<\infty)=1,\quad \mathrm{where}\quad \tau = \min \{k\geq 1\colon X_k = \partial\}.
\end{align}
By weak reversibility of $\cN$,   $X$ is irreducible, and thus \eqref{equ_sd2} is true because $\cX_0$ is a finite set, see Lemma \ref{lmm_fnit2}i).

\textbf{Step 2:}  $R(\eta) = R_0(\eta)$. Using the same idea as in Step 2 in Section \ref{prf_itms}, we can reduce this task to showing that for every reaction $r^* = y' + U_{\iota} \ce{->} \eta \in \cR_{\iota,0}$, it holds that
\begin{align}\label{equ_sdr}
\sum_{\gamma\in \Xi_0,\kR(\ter(\gamma))=r^*} \pi (x - \zeta_{\kR(\gamma)}) \lambda^*_{\gamma}(x - \zeta_{\kR(\gamma)}) = \pi (x - \zeta_{r^*}) \lambda_{r^*}(x - \zeta_{r^*})
\end{align}
for all $x\in \Gamma_0$ such that $x - \zeta_{r^*} \geq \reac (r^*)$. 
 
Let $Y$ be a DTMC on $\cY$ with transition probabilities
\[
P_{z,z'}=\begin{cases}
\displaystyle \sum_{r\in \cR_{0,0}}\frac{  \pi(x - \zeta_r) \lambda_{y\to z}(x - \zeta_r) }{\pi(x)\sum_{j = 0}^m \sum_{r'\in \cR_{0,j}} \lambda_{r'}(x) }, & z = z' = \partial,\\[12pt]
\displaystyle \sum_{r\in \cR_{i,0}\colon  \zeta_r = z' - x}\frac{\pi(z') \lambda_{r}(z') }{\pi(x)\sum_{j = 0}^m \sum_{r'\in \cR_{0,j}} \lambda_{r'}(x)  }, & z = \partial, \rho (z') = U_i, \\[12pt]
\displaystyle  \sum_{r\in \cR_{0, j}}\frac{\pi(z - \zeta_r) \lambda_{r}(z - \zeta_r ) }{\pi (z)\sum_{j'=0}^{m}\sum_{r'\in \cR_{j,j'}} \lambda_{r'}(z)  }, & \rho (z) = U_j , z' = \partial,\\[12pt]
\displaystyle  \sum_{r\in \cR_{i, j}\colon  \zeta_r = z' - z}\frac{\pi(z') \lambda_{r}(z' ) }{\pi (z)\sum_{j'=0}^{m}\sum_{r'\in \cR_{j,j'}} \lambda_{r'}(z)  }, & \rho (z) = U_j , z' = U_i,
\end{cases}
\]
where $i,j\in \{1,\dots, m_{\pro}\}$. Following the same line of thought as in Section \ref{prf_itms} step 2, we can show that \eqref{equ_sdr} is equivalent to $\mathbb{P}_{y' + U_i} (\tau < \infty) = 1$, where $\tau \coloneqq \min \{k \geq 1 \colon Y_k = \partial\}$. This is due to the recurrence of an irreducible DTMC in a finite state space. The proof of Theorem \ref{thm_sdpi0}ii) is complete.
\end{proof}

\subsection{Proof of Proposition \ref{prop_itms}}\label{prf_itms}

If $\cU$ consists of intermediate species, then   $\Gamma_0$ is an irreducible component of $\cN_{\cU}$  \cite[Theorem 5.6]{hoessly2021sum}. In addition, if $\gamma\in \Xi_0$, then $\kR(\gamma)=y\ce{->}y'$, where $y=\reac (r_{0})$ and $y'=\prdt (r_{q})$. Thus, $\cC_{\cU}\subseteq \cC_0=\{y\in \cC\colon \rho  (y) = 0\}$. To prove the proposition, it suffices to show the following equation for all $\eta\in \cC_{\cU}$ and $x\in\Gamma_0$,
\begin{align}\label{equ_cbmit0}
\sum_{r\in \out_\cU(\eta)} \pi(x)\lambda_{\cU,r}(x)=\sum_{r\in \inc_\cU(\eta)} \pi(x-\zeta_r)\lambda_{\cU,r}(x-\zeta_r),
\end{align}
where $\out_\cU (\eta)$ and $\inc_\cU(\eta)$ denotes the set of outgoing and incoming reactions of $\eta$ in $\cR_{\cU}$. 
If $r\in \kR(\Xi_0)\setminus\cR_\cU$, 
 then $r=(\eta,\eta)$ for some $\eta$, and $r$ is both an outgoing `reaction' of $\eta$, and an incoming `reaction' of $\eta$, with $\zeta_r=\eta-\eta=0$. It follows that \eqref{equ_cbmit0} is equivalent to 
\begin{align}\label{equ_cbmit}
\sum_{r\in \out_\cU^* (\eta)} \pi(x)\lambda_{\cU,r}(x)=\sum_{r\in \inc_\cU^* (\eta)} \pi(x-\zeta_r)\lambda_{\cU,r}(x-\zeta_r),
\end{align}
where $\out_\cU^* (\eta)$ and $\inc_\cU^*(\eta)$ denote the set of outgoing and incoming `reactions' of $\eta$ in $\kR(\Xi_0)$ 
and $\lambda_{\cU}$ is extended to $\kR(\Xi_0)$   
by \eqref{def_lmd'}. 

Recall that $\pi$ is a complex balanced distribution for $(\cN,\lambda)$ and $\eta$ as in \eqref{equ_cbmit} is in $\cC_{\cU}\subseteq \cC_0\subseteq \cC$. It follows that
\begin{align*}
\sum_{r\in \out (\eta)} \pi(x)\lambda_r(x)=\sum_{r\in \inc (\eta)} \pi(x-\zeta_r)\lambda_r(x-\zeta_r)
\end{align*}
where $\out (\eta)$ and $\inc (\eta)$ denotes the set of outgoing and incoming reactions of $\eta$ in $\cR$.

Denote by $L_0 (\eta)$ and $R_0 (\eta)$ the left-hand side and right-hand side of \eqref{equ_cbmit}; and by $L (\eta)$ and $R (\eta)$ the left-hand side and right-hand side of \eqref{equ_cbmit0}, respectively. If  $x\not\geq \eta$, then $L_0(\eta)=R_0(\eta)=0$. Thus, it suffices to show \eqref{equ_cbmit} assuming $x\geq \eta$. 

\textbf{Step 1:}  $L(\eta)=L_0(\eta)$. We have
\[
L(\eta)=\pi(x)\Big(\sum_{r\in \out(\eta)\cap \cR_{0,0}}\lambda_r(x)+\sum_{i = 1}^{m_{\pro}} \sum_{r\in \out(\eta)\cap \cR_{0,i}}\lambda_r(x)\Big)
\]
and
\[
L_0(\eta) = \pi(x)\sum_{\gamma\in \Xi_0 \colon g(\gamma)=\eta}\lambda^*_{\gamma}(x) = \pi(x) \Big(\sum_{r\in \out(\eta)\cap \cR_{0,0}}\lambda_r(x) + \sum_{k = 1}^{\infty} \sum_{\gamma\in \Xi_0 (k) \colon g(\gamma)=\eta}\lambda^*_{\gamma}(x)\Big),
\]
where $m_{\pro}$ is given in \eqref{def_mpro}, and $g(\gamma) \coloneqq \reac(\init(\gamma))$ is the reactant of the   reaction of the initial edge of $\gamma$. We claim that for any reaction $r_* = \eta\ce{->[$r_*$]}U_\iota$ with $U_\iota\in \cU_{\pro}$, it holds that
\begin{align}\label{equ_xi0}
\lambda_{r_*} (x) = \sum_{k = 1}^{\infty} \sum_{\gamma\in \Xi_0(k)\colon \init(\gamma)=U_0\overset{r_*}{\longrightarrow} U_{\iota}}\lambda^*_{\gamma}(x).
\end{align}
Then, by summing both sides of \eqref{equ_xi0} among $r_* \in \out (\eta)\setminus \cR_{0,0}$, we get
\begin{align*}
\sum_{r_*\in \out(\eta)\setminus \cR_{0,0}}\lambda_{r_*}(x)=\sum_{\gamma\in \Xi_0\colon g(\gamma)=\eta}\lambda^*_{\gamma}(x),
\end{align*}
and thus $L(\eta)=L_0(\eta)$. 

To prove \eqref{equ_xi0}, we construct a DTMC $X$ on $\cX=\{U_0\} \cup \cU_{\pro}$. Due to \eqref{equ_sbt=1},  the transition probabilities  can be defined by 
\[
\mathbb{P}_{U_i} (U_j) = P_{i,j} =
 \displaystyle  \sum_{r \in \cR_{i,j}} \beta_{r} \big(x - \eta + \reac (r)\big).
\]
Thus, for any $i,j\in\{0,\dots, m_{\pro}\}$, $P_{i,j}>0$ if and only if $U_i \ce{->} U_j \in \cR$.
Additionally, since $\cU$ is eliminable, it follows that $X$ is irreducible. Denote by $\tau=\min\{k\geq 1\colon X_k=U_0\}$. Note that
\begin{align*}
  \mathbb{P}_{U_{\iota}} (\tau = 1) = P_{\iota, 0}   &= \sum_{r \in \cR_{\iota, 0}} \beta_r (x - \eta + U_{\iota}) =  \lambda_{r_*}(x)^{-1} \sum_{r \in \cR_{\iota, 0}}  \lambda_{r_*}(x) \beta_r (x - \eta + U_{\iota})\\
    &= \lambda_{r_*}(x)^{-1} \sum_{\gamma \in \Xi_0(1)\colon \init (\gamma) = U_0\overset{r_*}{\longrightarrow} U_{\iota} } \lambda_{\gamma}^*(x).
\end{align*}
By iteration, one can show that for every $k \in \N$,
\begin{align*}
 \mathbb{P}_{U_{\iota}} (\tau = k) = \lambda_{r_*}(x)^{-1} \sum_{\gamma \in \Xi_0(k)\colon \init (\gamma) = U_0\overset{r_*}{\longrightarrow} U_{\iota} } \lambda_{\gamma}^*(x).
\end{align*}
 Taking \eqref{def_lmd0} into account and using that the Markov chain is  irreducible   (cf. \cite[Theorems 1.5.6 and 1.5.7]{norris}), we have
\begin{align*}
1 = \mathbb{P}_{U_{\iota}}(\tau<\infty) = \sum_{k=1}^{\infty}\mathbb{P}_{U_{\iota}}(\tau = k) = \lambda_{r_*} (x)^{-1} \sum_{k=1}^{\infty} \sum_{\gamma\in \Xi_0 (k)\colon \init(\gamma) = U_0\overset{r_*}{\longrightarrow} U_{\iota}}\lambda_{\gamma}^*(x).
\end{align*}
This proves \eqref{equ_xi0} and thus $L(\eta)=L_0(\eta)$.

\textbf{Step 2:} $R(\eta)=R_0(\eta)$.
Similarly, in order to show that $R(\eta)=R_0(\eta)$, it suffices to verify that for any $r^{*} = U_{\iota} \ce{->[$r^*$]} \eta$ with some $U_{\iota} \in \cU_{\deg} = \cU_{\pro}$ (see Lemma \ref{lmm_i=j}), the next equation holds,
\begin{align}\label{equ_xi01}
\pi(x-\eta+U_{\iota})\lambda_{r^*} (x-\eta+U_{\iota}) = \sum_{k=1}^{\infty} \sum_{\gamma\in \Xi_0(k) \colon \ter(\gamma) = U_{\iota} \overset{r^*}{\longrightarrow} U_0}\pi(x-\zeta_{\kR(\gamma)})\lambda_{\gamma}^*(x-\zeta_{\kR(\gamma)}).
\end{align}

For any $i,j \in \{0,\dots, m_{\pro}\}$, let
\[
P_{i,j}=\begin{cases}
\displaystyle \frac{\sum_{y \colon y \to \eta \in \cR_{j,0}}\pi(x - \eta + y) \lambda_{y \to \eta}(x - \eta + y) }{\pi (x) \sum_{y' \colon \eta \to y'\in \cR } \lambda_{\eta \to y'}(x) }, & i = 0,\\[12pt]
\displaystyle \frac{ \sum_{r\in \cR_{j,i}} \pi(x - \eta + \reac(r)) \lambda_{r}(x - \eta + \reac(r)) }{\pi (x - \eta + U_i)\sum_{y'\colon U_i \to y' \in \cR} \lambda_{U_i \to y' }(x - \eta + U_i)  }, & i \neq 0. \\[12pt]
\end{cases}
\]
Recall that $\pi$ is a complex balanced distribution for $(\cN,\lambda)$ on $\cN$, then due to \eqref{def_cbd} and Condition \ref{con_cmp}, we see that $\mathbb{P}_{U_i} (U_j) = P_{i,j}$, $i,j \in \{0,\dots, m_{\pro}\}$ forms a transition probability for a DTMC, say $Y$, on $\cX$.
Let $\tau \coloneqq \min\{q\geq 1\colon Y_q = U_0\}$, and for any $k\in \N$, let
\begin{align*}
R_{1, k}(\eta)\coloneqq \sum_{\gamma\in \Xi_0(k) \colon \ter(\gamma)= U_{\iota} \overset{r^*}{\longrightarrow} U_0}\pi(x-\zeta_{\kR(\gamma)})\lambda_{\gamma}^*(x-\zeta_{\kR(\gamma)}).
\end{align*}
Then, we aim to show that
\begin{align}\label{eq_rtint}
  \mathbb{P}_{U_{\iota}} (\tau = k) = R_{1,k}(\eta) / L_1(\eta),
\end{align}
for all $k \geq 1$, where $L_1(\eta)$ is the left-hand side of equation \eqref{equ_xi01}.
Suppose $k = 1$. Then,
\begin{align}\label{eq_tauy=1}
 \mathbb{P}_{U_{\iota}} (\tau = 1) =  P_{\iota,0} = & \frac{ \sum_{r\in \cR_{0,\iota}} \pi(x - \eta + \reac(r)) \lambda_{r}(x- \eta + \reac(r)) }{\pi (x - \eta + U_{\iota}) \sum_{y'\colon U_i \to y' \in \cR} \lambda_{U_i \to y' }(x - \eta + U_i) }.
\end{align}

On the other hand, for any  $r' = y\ce{->} y'\in \cR_{i,j}$ with some $i,j\in \{0,\dots, m_{\pro}\}$ and $x'\geq y$, it follows from \eqref{def_btijk} and \eqref{equ_sbt=1},
\begin{align*}
\lambda_{r'}(x') = \lambda_{r'}(x')\sum_{j'=0}^{m}\sum_{r''\in \cR_{i,j'}}\beta_{r''}(x') = \beta_{r'}(x') \sum_{j'=0}^{m}\sum_{r'' \in \cR_{i,j'}}\lambda_{r''}(x') > 0,
\end{align*}
and therefore,
\begin{align}\label{equ_btlmd}
\beta_{r'}(x')  = \lambda_{r'}(x') \bigg/  \sum_{j'=0}^{m}\sum_{r'' \in \cR_{i,j'}}\lambda_{r''}(x').
\end{align}
As a consequence, we have
\begin{align}\label{eq_r11}
R_{1,1}(\eta) = & \sum_{y \colon y\to U_{\iota} \in \cR_{0,\iota}} \pi(x - \eta + y) \lambda_{y \to U_{\iota}} (x-\eta + y) \beta_{r^*} (x-\eta + y - y + U_{\iota}) \nonumber\\
= &   \frac{ \lambda_{r^*} (x-\eta + U_{\iota})\sum_{y \colon y\to U_{\iota} \in \cR_{0,\iota}} \pi(x-\eta + y) \lambda_{y \to U_{\iota}} (x-\eta + y) }{ \sum_{y' \colon U_{\iota} \to y' \in \cR}\lambda_{U_{\iota} \to y'}(x-\eta + U_{\iota})}.
\end{align}
By comparing \eqref{equ_xi01}, \eqref{eq_tauy=1} and \eqref{eq_r11}, one obtain equation \eqref{eq_rtint} with $k = 1$.

Similarly, assume $k = 2$. Then, we can write
\begin{align*}
\mathbb{P}_{U_{\iota}} (\tau = 2) =  \sum_{i=1}^{m_{\pro}} \sum_{y  \colon y \to U_{i}\in \cR_{0,i} } \bigg[ &\frac{\pi(x+\eta-y) \lambda_{y\to U_{i}}(x - \eta + y) }{\sum_{y' \colon U_i \to y'\in\cR} \lambda_{U_i \to y'}(x)  }\\
&\times \frac{ \lambda_{U_{i} \to U_{\iota}}(x- \eta + U_i) }{\pi(x-\eta+U_{\iota})\sum_{y' \colon U_{\iota} \to y' \in \cR}\lambda_{U_{\iota} \to y'}(x-\eta + U_{\iota})}\bigg],
\end{align*}
and using \eqref{equ_btlmd},
\begin{align*}
R_{1,2}(\eta) = & \sum_{i=1}^{m_{\pro}}  \sum_{y  \colon y \to U_{i}\in \cR_{0,i} } \pi(x-\eta + y) \lambda_{y \to U_{i}} (x-\eta + y) \beta_{U_i\to U_{\iota}} (x-\eta + U_i) \beta_{r^*} (x-\eta + U_{\iota}) \\
= & \bigg[\sum_{i=1}^{m_{\pro}} \sum_{y  \colon y \to U_{i}\in \cR_{0,i} } \frac{  \lambda_{y \to U_{i}} (x-\eta + y) \lambda_{U_i\to U_{\iota}} (x-\eta + U_i) }{  \sum_{y' \colon U_{i} \to y' \in \cR}\lambda_{U_i \to y'}(x-\eta + U_i)}\bigg]\\
&\times \frac{ \lambda_{U_{\iota}\to \eta}(x-\eta + U_{\iota})  }{ \sum_{y' \colon U_{\iota} \to y' \in \cR}\lambda_{U_{\iota} \to y'}(x-\eta + U_{\iota}) }.
\end{align*}
This implies that equation \eqref{eq_rtint} holds with $k = 2$ as well.
By using an iteration argument, one should be convinced that \eqref{eq_rtint} is true for all $k \in \N$. This leads to \eqref{equ_xi01} and hence $R_0(\eta) = R(\eta)$. The proof of this proposition is complete.

\subsection{Proof of Proposition \ref{prop_rvsb}}\label{prf_rvsb}
The   reduced reaction network $\cN_{\cU}$ is reversible  \cite[Theorem 5.1]{hoessly2021sum}. For   $\gamma\in \Xi_0$, by reversibility, there exists the walk in the opposite direction
\[ 
\gamma'=U_0\ce{->[$r_q'$]}U_{i_q}\ce{->}\cdots \ce{->[$r_1'$]}U_1\ce{->[$r_0'$]}U_0\ \ \in \Xi_0,
\]
where $r_j' =y'\ce{->}y$ if $r_j = y\ce{->}y'$ for   $0\leq j\leq q$ with   $i_0=i_{q+1}=0$, . This implies $\zeta_{r_j} = -\zeta_{r_j'}$ for all $0\leq j\leq q$, and thus $\zeta_{\gamma} = -\zeta_{\gamma'}$ with $\zeta_{\gamma} \coloneqq \sum_{j=0}^q\zeta_{r_j}$.
For any $x\in \Gamma_0$, define
\begin{align*}
x'\coloneqq x+\zeta_{\gamma} = x + \sum_{j=0}^q \zeta_{r_j} = x - \sum_{j=0}^q \zeta_{r_j'} = x-\zeta_{\gamma'}.
\end{align*}
By \eqref{def_lmd'} it is enough to verify the following equation
\begin{align}\label{eq_db}
\pi(x)\lambda_{\gamma}^* (x) = \pi (x') \lambda_{\gamma'}^* (x'),
\end{align}
provided that $x\geq \reac(\kR(\gamma))$.  Recall that $\pi$ is a detailed balanced distribution for $(\cN,\lambda)$. By definition,  
\[
\pi(x)\lambda_{r_j}(x) = \pi(x-\zeta_{r_j'})\lambda_{r_j'}(x-\zeta_{r_j'}) = \pi(x + \zeta_{r_j})\lambda_{r_j'}(x+\zeta_{r_j}).
\]
Therefore,
\begin{align*}
\pi(x)\lambda^*_{\gamma}(x)&=\pi(x)\lambda_{r_0}(x)\prod_{j=1}^{q}\frac{\lambda_{r_j}(x+\zeta_{r_0}+\dots+\zeta_{r_{j-1}})}{\sum_{j'=0}^m\sum_{r \in \cR_{i_j,j'}} \lambda_{r} (x+\zeta_{r_0}+\dots+\zeta_{r_{j-1}})}\\
&=\lambda_{r_0'}(x+\zeta_{r_0}) \pi(x + \zeta_{r_0}) \frac{\lambda_{r_1}(x+\zeta_{r_0})}{\sum_{j'=0}^m\sum_{r'\in \cR_{i_1,j'}}\lambda_{r'}(x+\zeta_{r_0})}\\
&\quad\times\prod_{j=2}^{q}\frac{\lambda_{r_j}(x+\zeta_{r_0}+\dots+\zeta_{r_{j-1}})}{\sum_{j'=0}^m\sum_{r''\in \cR_{i_j,j'}} \lambda_{r''}(x+\zeta_{r_0}+\dots+\zeta_{r_{j-1}})}\\
&=\beta_{r_0'}(x'+\zeta_{r_q'}+\dots+\zeta_{r_1'})\pi(x+\zeta_{r_0})\lambda_{r_1} (x+\zeta_{r_0})\\
&\quad \times\prod_{j=2}^{q}\frac{\lambda_{r_j}(x+\zeta_{r_0}+\dots+\zeta_{r_{j-1}})}{\sum_{j'=0}^m\sum_{r''\in\cR_{i_j,j'}}\lambda_{r''}(x+\zeta_{r_0}+\dots+\zeta_{r_{j-1}})}.
\end{align*}
Similarly, we deduce that
\begin{align*}
& \pi(x+\zeta_{r_0})\lambda_{r_1} (x+\zeta_{r_0}) \prod_{j=2}^{q}\frac{\lambda_{r_j}(x+\zeta_{r_0}+\dots+\zeta_{r_{j-1}})}{\sum_{j'=0}^m\sum_{r \in \cR_{i_j,j'}}\lambda_{r}(x+\zeta_{r_0}+\dots+\zeta_{r_{j-1}})}\\
&= \beta_{r_1'}(x'+\zeta_{r_q'}+\dots+\zeta_{r_2'})\pi(x+\zeta_{r_0}+\zeta_{r_1})\lambda_{r_2}(x+\zeta_{r_0}+\zeta_{r_1})\\
&\quad\times\prod_{j=3}^{q}\frac{\lambda_{r_j}(x+\zeta_{r_0}+\dots+\zeta_{r_{j-1}})}{\sum_{j'=0}^m\sum_{r'\in \cR_{i_j,j'}}\lambda_{r'}(x+\zeta_{r_0}+\dots+\zeta_{r_{j-1}})}
\end{align*}
Repeating the above arguments, one   ends up with the   equation,
\begin{align*}
\pi(x)\lambda^*_{\gamma}(x) = & \pi(x')\lambda_{\gamma'}^*(x')\prod_{j=1}^q \beta_{r_{j-1}'} \Big(x' + \sum_{j'=j}^{q} \zeta_{r_{j'}'} \Big) = \pi(x')\lambda_{\gamma'}^* (x').
\end{align*}
This proves  \eqref{eq_db}, and thus $\pi_0$ is a detailed balanced distribution (or linear combination of detailed balanced distributions) for $(\cN_{\cU},\lambda_{\cU})$ on $\Gamma_0$. The proof is complete.

\subsection{Proof of Lemma \ref{lmm_fnit2}}\label{app_fnit2}

It suffices to show Lemma \ref{lmm_fnit2}i) by proving statements
\begin{enumerate}
\item[A.] If Condition \ref{con_2}i) fails, then $\cX_0$ in \eqref{def_cx} is infinite for some $x\in \N_0^n$.

\item[B.] If Condition \ref{con_2}i) holds, then $\cX_0$ in \eqref{def_cx} is finite for all $x\in \N_0^n$.
\end{enumerate}

\begin{proof}[Proof of Statement A]
 Suppose Condition \ref{con_2}i) fails, then, there exists $\gamma \in \Xi_i$, such that $\kR (\gamma) = (y + U_i, y' + U_i) \in \Xi_i$ for some $U_i\in \cU_{\pro}$ and $y < y'$. Since  $U_i\in \cU_{\pro}$, then there exists a walk $\theta$ in $(\cV,\cE)$ with $\init (\theta) = U_0$ and $\ter (\theta) = U_i$. Consider the following sequence of paths
\[
\theta_q \coloneqq \theta  +  \underbrace{\gamma + \cdots + \gamma}_{\text{$q$ instances of $\gamma$}} \in \Theta_0,\qquad q\geq 1.
\]  
Let $\kR(\theta) = (y_0, y_0')$. Then, 
\begin{align*}
(z_q,z_q') \coloneqq \kR(\theta_q) = \kR(\theta) \oplus \underbrace{\kR(\gamma) \oplus \cdots \oplus \kR(\gamma)}_{\text{$q$ instances of $\kR(\gamma)$}} = (y_0,y_0') \oplus (y, y+q(y'-y)).
\end{align*}
Using \cite[Proposition 2(4)]{hoessly2021sum}, we have $z_q\leq y_0+y_1$ and $z_q' \geq y+q(y'-y)$. It follows that the set
\begin{align*}
\{x + \zeta_{\kR(\theta_q)} = y_0 + z_q' - z_q\colon q\in \N\}
\end{align*}
is an infinite subset of $\cX_0$. This completes the proof of Statement A.
\end{proof}

The proof of Statement B is based on the next two lemmas.

\begin{lemma}\label{lmm_posiseq}
Let $n$ be a positive integer and let $\{x_n\}_{n\geq 1} \subseteq \R^n$ be a sequence of vectors that is bounded from below. Then, there exists a non-decreasing sub-sequence $\{x_{n_k}\}_{k\geq 1}$ of $\{x_n\}_{n \geq 1}$. Moreover, if $x_n \neq x_m$ for all   $n\neq m$, $\{x_{n_k}\}_{k\geq 1}$ is strictly increasing. 
\end{lemma}

\begin{lemma}\label{lmm_tc1}
Let $x\in \N_0^n$, $q$ be a positive integer, and $\cZ \coloneqq \{\zeta_1,\dots, \zeta_q\}\subseteq \Z^n\setminus \N_0^n$. Let $\cH$ consist of all linear combination of elements in $\cZ$  with non-negative integer coefficients,
\[
\cH \coloneqq \bigg\{\sum_{i=1}^q k_i \zeta_i\colon k=(k_1,\dots,k_q)\in \N_0^q\bigg\}.
\]
Then, i) $\cH_x \coloneqq \{\eta\in \cH\colon x+\eta\in \N_0^n\}$ is a finite set, if and only if ii) $\cH\cap \N_0^n\setminus\{0\}=\emptyset$.
\end{lemma}

\noindent
{\bf Remark.} We believe the lemma is in the literature but we have not been able to find it, hence we provide a proof.

\begin{proof}[Proof of Statement B]
Suppose Condition \ref{con_2}i) to hold. For any $\theta\in \Theta_0$, there exists a collection of closed walks $\{\gamma_i\}_{i=1}^q$  connected by paths $\{\theta_i\}_{i=1}^{q+1}$, such that
\[
\theta = \theta_1 + \gamma_1 + \theta_2 + \dots + \gamma_{q} + \theta_{q+1},
\]
where $\gamma_i\in \Xi_{j_i}$ with  $j_i\in \{1,\dots, m_{\pro}\}$, $i=1,\dots, q$, such that $j_i\neq j_{i'}$ for all $i\neq i'$. Then, $q\leq m_{\pro}$. As the number of paths   in $(\cV,\cE)$ is finite, it suffices to show that the set
\begin{align*}
A_i(x)\coloneqq \{x + \zeta_{\kR(\gamma)} \colon \gamma \in \Xi_i, \reac (\kR(\gamma)) \leq x\}
\end{align*}
is finite for every $x\in \N_0^n$ and $i\in \{1,\dots, m_{\pro}\}$. For any $i\in\cC$, let
\[
A_{i,k}(x) = \{ x + \zeta_{\kR(\gamma)} \colon \gamma \in \Xi_i, x\geq \reac (\kR(\gamma))\}, |\cC_{\gamma}| = k \},
\]
with $k\in\{1,\dots, m_{\pro}\}$ and $\cC_{\gamma} = \{U_i\in \cU \colon U_i\in\gamma\}$.
 Then, $A_i(x) = \cup_{k=1}^{m_{\pro}} A_{i,k} (x)$, $i\in \{1,\dots, m_{\pro}\}$. We aim to show that $A_{i,k} (x)$ is a finite set for all $i\in \{1,\dots, m_{\pro}\}$, $1\leq k\leq m_{\pro}$ and $x\in \N_0^n$ by induction in $k$.

Fix any $i_0\in \{1,\dots, m_{\pro}\}$. Note that $\{\Gamma\in \Xi_{i_0}\colon |\cC_{\gamma}| = 1\} = \cR_{i_0,i_0}$ is a finite set. Thus, due to Condition \ref{con_2}i) and Lemma \ref{lmm_tc1},   $A_{i_0,1}(x)$ is finite. For $k\geq 2$, we show  that $A_{i_0,k}(x)$ is finite by contradiction. Suppose $A_{i_0,k}(x)$ is not a finite set. Then, there exists a sequence of closed walks $\{\gamma_j\}_{j\geq 1} \subseteq \Xi_{i}$, such that
\begin{enumerate}
\item[a)]  $|\cC_{\gamma_j}| = k$, $x\geq \reac (\kR (\gamma_j))$ for all $i\in \N$,

\item[b)] $\{x + \zeta_{\kR(\gamma_j)}\}_{j\geq 1} \subseteq A_{i,k}$ is an infinite subset of $\N_0^n$,

\item[c)] $ \zeta_{\kR (\gamma_j)} \neq \zeta_{\kR (\gamma_{j'})}$ for all indexes $j\neq j'$. 
\end{enumerate}

Using Lemma \ref{lmm_posiseq}, we can assume that $\{\zeta_{\kR(\gamma_j)}\}_{j\geq 1}$ is a strictly increasing sequence. Next, for every $j\in \N$,   $\gamma_j$ can be decomposed as $\gamma_j = \gamma_{j,1} + \dots + \gamma_{j,q_j}$, such that for   $j'=1,\dots, q_j$, $\gamma_{j,j'}\in \Xi_{i_0}$ and $U_{i_0}$ only appears as the initial and terminal nodes  of $\gamma_{j,j'}$. 
Assume that 
\begin{align}\label{ass_neq}
\sum_{j'=k_1}^{k_2} \zeta_{\kR(\gamma_{j,j'})} \neq 0,\quad \forall 1\leq k_1\leq k_2\leq q_i.
\end{align}
 Otherwise, we can replace $\gamma_j$ by $\gamma_j' = \gamma_{j,1} + \dots + \gamma_{j,k_1 - 1} + \gamma_{j,k_2 + 1}$. Then, $\reac (\gamma_j') \leq \reac(\gamma_j) \leq x$ and $\zeta_{\kR(\gamma_j')} = \zeta_{\kR(\gamma_j)}$. Thus, with this substitution, properties a)--c) still hold. 

For any $i\in \N$, $\gamma_{i,1}$ can be decomposed into at most $k-1$ closed walks connected by paths. Due to the induction hypothesis, for any $x\in\N_0^n$, the set $\{x + \zeta_{\kR(\gamma_{i,1})}\}_{i\geq 1}$ is finite. Thus, for any positive integer $q$ and any index $j$, $\{x + \zeta_{\kR(\gamma_{j,1})} + \dots + \zeta_{\kR(\gamma_{j,j'})}\}_{j'=1}^{q_j\wedge q}$ is a finite set. Therefore, if $\{q_j\}_{j\geq 1}$ is finite, then so is $\{x + \zeta_{\kR(\gamma_j)}\}_{j\geq 1}$. This contradicts b). Thus, by taking sub-sequences, if necessary, we can assume that $\{q_j\}_{j\geq 1}$ is a strictly increasing sequence of positive integers.
Denote 
\[
\eta_{j,j'} \coloneqq x + \zeta_{\kR(\gamma_{j,1})} + \dots + \zeta_{\kR(\gamma_{j,j'})},
\]
for all $j\in \N$ and $j' \in \{1,\dots, q_j\}$. It follows from a) that $\eta_{j,j'} \in \N_0^n$. 
Recall that $\{\eta_{j,1} \}_{j\geq 1} = \{x + \zeta_{\kR(\gamma_{j,1})}\}_{j\geq 1}$ is a finite set. There exist a sub-sequence $\{\eta_{n_{j}^1,1}\}_{j\geq 1}$ of $\{\eta_{j,1}\}_{j\geq 1}$ such that $\eta_{n_j^1,1}  = y_1\in\N_0^n$ for all indexes $j'$. Iteratively, for any $k\geq 2$, there exists a sub-sequence
$\{\eta_{n_j^k,k} \}_{j\geq 1}$ of $\{\eta_{n_j^{k-1},k} \}_{j\geq 1}$,
such that 
$\eta_{n_j^k,k} = y_{k}\in\N_0^n$ 
for all  $j\in \N$. Concerning Condition \ref{con_2}i),  the sequence $\{y_k\}_{k\geq 1}\subseteq \N_0^n$    satisfies the property that 
 \begin{align}\label{yk_nincr}
y_{k} - y_{k'} = \eta_{n_{k}^{k}, k} - \eta_{n_{k}^{k}, k'} = \sum_{j=k'+1}^{k} \zeta_{\kR(\gamma_{n_{k}^k, j})} \notin \N_0^n\setminus\{0\},\quad \forall k > k' \geq 1.
\end{align}
On the other hand, as a result of assumption \eqref{ass_neq}, we have $y_{k} \neq y_{k'}$ for all $k\neq k'$. Thus, according to Lemma \ref{lmm_posiseq}, there exists a strictly increasing sub-sequence of $\{y_k\}_{k\geq 1}$, which contradicts  \eqref{yk_nincr}. This proves that  $A_{i,k}(x)$ is a finite set for all $i\in \{1,\dots, m_{\pro}\}$ and $1\leq k\leq m_{\pro}$. Therefore, $A_i(x) = \cup_{k=1}^{m_{\pro}} A_{i,k}(x)$ is also finite, and so is $\cX_0$. The proof of Statement B is complete.
\end{proof}

Now it suffices to show Lemmas \ref{lmm_posiseq} and \ref{lmm_tc1}. Note that Lemma \ref{lmm_posiseq} is elementary, therefore we only provide the proof of Lemma \ref{lmm_tc1}.

\begin{proof}[Proof of Lemma \ref{lmm_tc1}]
i)$\implies$ii) is straightforward: If ii) fails, then there exists $\eta\in \cH\cap \N_0^n\setminus\{0\}$. This implies, $\cG \coloneqq \{j\eta\colon j\in \N\}$ is an infinite subset of $\cH_x$, which contradicts i).

Conversely, suppose ii) holds. Let $V\subseteq \N_0^q$ be a set such that $\cH_{x,V} \coloneqq \{x+v(\cZ) \colon v\in V\}=\cH_x$.
By Zorn's lemma, we can assume that $V$ is minimal, namely, for any $v\in V$, with $V'=V\setminus \{v\}$, it holds that $\cH_{x,V'}\neq \cH_x$. To prove $H_x$ is finite, it suffices to show that $V$ is a finite set. Suppose that $V$ is an infinite set. Then, there is  a sequence $\{v_i\}_{i \geq 1} \subseteq V$, such that $v_i\neq v_j$ for any $i\neq j$. 
Concerning Lemma \ref{lmm_posiseq}, by taking sub-sequence, we can assume that $\{v_i\}_{i\geq 1}$ is strictly increasing.

For any $\eta\in \cH_x$, we have $\eta \geq -x$. Thus, $\{v_i(\cZ)\}_{i\geq 1}\subseteq \cH_x$ is bounded from  below.  Due to the minimality of $V$, we have $v_i(\cZ)\neq v_j(\cZ)$, whenever $i\neq j$, and thus by Lemma \ref{lmm_posiseq}, assume  $\{v_i(\cZ)\}_{i\geq 1}$ is strictly increasing. Because $\{v_i\}_{i\geq 1}$ is strictly increasing, it follows that $v_0\coloneqq v_{2}-v_{1} \in \N_0^{q,*}$, and
$ v_0(\cZ)=v_{2}(\cZ)-v_{1}(\cZ) \in \cH\cap \N_0^n\setminus\{0\}$.
This contradicts ii).  The proof is complete.
\end{proof}

\end{document}